\documentclass[12pt,a4paper,reqno]{amsart}
\usepackage[T1]{fontenc}
\usepackage{lmodern}
\usepackage{amsmath,amssymb,mathtools,mathrsfs}
\usepackage{microtype}
\usepackage[a4paper,hmargin=28mm,top=26mm,bottom=28mm,includeheadfoot]{geometry}
\usepackage{tikz-cd}
\usepackage{booktabs}
\usepackage{xurl}
\usepackage{xcolor}
\definecolor{linkblue}{RGB}{35,64,105}
\usepackage[colorlinks=true,linkcolor=linkblue,citecolor=linkblue,urlcolor=linkblue]{hyperref}

\allowdisplaybreaks[1]

\newtheoremstyle{readableplain}
  {9pt plus 2pt minus 1pt}{8pt plus 2pt minus 1pt}
  {\itshape}{}{\bfseries}{.}{.5em}{}
\newtheoremstyle{readabledefinition}
  {9pt plus 2pt minus 1pt}{8pt plus 2pt minus 1pt}
  {\normalfont}{}{\bfseries}{.}{.5em}{}
\newtheoremstyle{readableremark}
  {9pt plus 2pt minus 1pt}{8pt plus 2pt minus 1pt}
  {\normalfont}{}{\itshape}{.}{.5em}{}
\theoremstyle{readableplain}
\newtheorem{theorem}{Theorem}[section]
\newtheorem{proposition}[theorem]{Proposition}
\newtheorem{lemma}[theorem]{Lemma}
\newtheorem{corollary}[theorem]{Corollary}
\newtheorem{problem}{Problem}[section]
\newtheorem{conjecture}{Conjecture}[section]
\theoremstyle{readabledefinition}

\theoremstyle{readableremark}
\newtheorem{remark}[theorem]{Remark}

\makeatletter
\renewcommand{\section}{\@startsection{section}{1}{\z@}
  {-2.5ex\@plus-.8ex\@minus-.2ex}{1.2ex\@plus.2ex}
  {\normalfont\large\bfseries\raggedright}}
\renewcommand{\subsection}{\@startsection{subsection}{2}{\z@}
  {-2ex\@plus-.6ex\@minus-.2ex}{.8ex\@plus.2ex}
  {\normalfont\bfseries\raggedright}}
\renewcommand{\subsubsection}{\@startsection{subsubsection}{3}{\z@}
  {-1.5ex\@plus-.4ex\@minus-.2ex}{.6ex\@plus.1ex}
  {\normalfont\bfseries\raggedright}}
\makeatother

\numberwithin{equation}{section}

\DeclareMathOperator{\rk}{rk}

\newcommand{\Id}{\operatorname{Id}}

\title[Stable ample vector bundles]{Stable ample vector bundles without semipositive Kobayashi curvature}
\author{Kefeng Liu}

\address{Kefeng Liu: Mathematical Sciences Research Center, Chongqing University of Technology, Chongqing 400054, China}
\email{kefengliu@cqut.edu.cn}
\author{Xueyuan Wan}
\address{Xueyuan Wan: Mathematical Science Research Center, Chongqing University of Technology,
Chongqing 400054, China}
\email{xwan@cqut.edu.cn}
\makeatletter
\@namedef{subjclassname@2020}{\textup{2020} Mathematics Subject Classification}
\makeatother
\date{}

\begin{document}
\begin{abstract}
We construct slope-stable ample rank-two vector bundles on every
Hirzebruch surface $\mathbb{F}_e$ that admit no smooth strongly pseudoconvex
Finsler metric with semipositive Kobayashi curvature. These examples
give a negative answer to Kobayashi's problem. The same bundles
give counterexamples to the smooth strongly pseudoconvex Finsler
formulation of Fulton's problem and to the Hartshorne--Schneider
conjecture on the convexity of complements. For $e=1$, these
counterexamples have a simply connected
Fano base that is not a product.
\end{abstract}
\maketitle

\section{Introduction}

Ample vector bundles arise naturally in algebraic geometry,
in the study of embeddings, vanishing theorems, and the topology
of algebraic varieties. Understanding their relation to curvature
is also a basic problem in complex geometry. Following Hartshorne
\cite{Hartshorne66}, a holomorphic vector bundle $E$ is ample
when the tautological line bundle
$\mathcal O_{\mathbb P(E^*)}(1)$ is ample. Here and throughout
the paper, $\mathbb P(E)$ parametrizes one-dimensional subspaces
of the fibres of $E$. Curvature gives a differential-geometric
approach to this algebraic notion: a Griffiths-positive
Hermitian metric implies ampleness \cite{Griffiths69}. Recently,
Du--Xie \cite{DX} constructed a counterexample to the Griffiths
conjecture: an ample rank-two vector bundle on an abelian surface
admitting no Griffiths-semipositive Hermitian metric.

A Hermitian squared norm is quadratic in the fibre variables,
whereas a complex Finsler metric allows nonquadratic dependence on the
fibre direction. Strong pseudoconvexity requires positivity of
the complex fibre Hessian of the squared metric, but implies
neither real convexity nor the triangle inequality.
An obstruction to Hermitian metrics therefore does not by itself
exclude strongly pseudoconvex Finsler metrics.

The connection between complex Finsler metrics and ample bundles
is precise: the positivity of a tautological line bundle can be
expressed through curvature and plurisubharmonicity on the total
space of a vector bundle. Kobayashi proved that a holomorphic
vector bundle $E$ over a compact complex
manifold is ample if and only if its dual
$E^*$ admits a smooth strongly pseudoconvex complex Finsler
metric with negative Kobayashi curvature
\cite[Theorem~5.1]{Kobayashi75}. This characterization makes
Finsler metrics a natural tool for studying algebraic positivity
by methods of complex analysis. It uses a metric on the dual
bundle, and Kobayashi asked whether positive curvature on $E$
itself gives an equivalent
characterization \cite[\S5, Problem, p.~162]{Kobayashi75}; see also
\cite{Kobayashi96}.

\begin{problem}[Kobayashi]
\label{prob:kobayashi}
It is reasonable to expect that $E$ is ample if and only if it
admits a complex Finsler structure of positive curvature. The
question is whether $E$ admits a complex Finsler structure of
positive curvature if and only if $E^*$ admits a complex Finsler
structure of negative curvature.
\end{problem}

The implication from positive Kobayashi curvature to ampleness
was proved by Feng, Liu and Wan \cite{FLW}. We show that the
converse fails even when positive curvature is weakened to
semipositive curvature. The obstruction applies to every smooth
strongly pseudoconvex complex Finsler metric, without assuming
real convexity in the fibres. Our examples have rank two, are
slope-stable, and occur on every Hirzebruch surface.

For an integer $e\ge0$, the Hirzebruch surface is the smooth
projective rational surface
\[
 \mathbb F_e=
 \mathbb P_{\mathbb P^1}
       \bigl(\mathcal O_{\mathbb P^1}\oplus
                    \mathcal O_{\mathbb P^1}(e)\bigr),
 \qquad \pi:\mathbb F_e\longrightarrow\mathbb P^1.
\]
Let $f$ be the divisor class of a fibre of $\pi$, and let $S$
be the section defined by the line subbundle
$\mathcal O_{\mathbb P^1}(e)$ of the displayed direct sum.
Thus $S^2=-e$, $S\cdot f=1$, $f^2=0$.
In particular, $\mathbb F_0\simeq\mathbb P^1\times\mathbb P^1$;
for $e>0$, $S$ is the negative section. We use the standard
description of these surfaces and their ample divisors in
\cite[Chapter~V, \S2]{Hartshorne77}.

\begin{theorem}\label{thm:hirz-main}
For every integer $e\ge0$, put $H=S+(3e+1)f$ and $P=4H-f$.
There exists an ample $\mu_P$-stable holomorphic vector bundle
$E$ of rank two on $\mathbb F_e$, with $c_1(E)=4H$ and
$c_2(E)=13H^2=65e+26$, which admits no smooth strongly
pseudoconvex complex Finsler metric with semipositive Kobayashi
curvature.
\end{theorem}

Theorem~\ref{thm:hirz-main} answers
Problem~\ref{prob:kobayashi} negatively. Thus ampleness does
not guarantee even semipositive Kobayashi curvature on the
bundle itself. The bases are
simply connected, and for $e\ge1$ they are not products of
positive-dimensional compact complex manifolds.

We next consider two geometric properties suggested by the
curvature of positive Hermitian bundles. The second problem
concerns critical points of squared bundle norms.
Let $E$ be a holomorphic vector bundle of rank $r$
over a complex manifold $X$, equipped with a Griffiths-positive
Hermitian metric $h$, and write $G_h(v)=h(v,v)$.
Fulton's critical-point lemma \cite[p.~26]{Fulton87} states
that, for every locally closed complex submanifold
$\mathcal S\subset E\setminus\{0\}$ of dimension $k$, the
Levi form of $G_h|_{\mathcal S}$ at each critical point is
negative definite on a complex subspace of dimension at least
$k-r+1$. The lemma gives this bound whenever a critical point
exists, without requiring nondegeneracy.

Fulton then asks \cite[p.~28]{Fulton87}:
``Is this lemma true for a bundle which is known only to be ample?''
He suggests that a Finsler metric might replace the positive
Hermitian metric. We consider the following formulation for
smooth strongly pseudoconvex complex Finsler metrics, allowing
the metric to depend on the chosen submanifold.

\begin{problem}[Fulton, smooth strongly pseudoconvex Finsler formulation]
\label{prob:fulton}
Let $E$ be an ample holomorphic vector bundle of rank $r$ on
a smooth projective variety $X$, and let
$\mathcal S\subset E\setminus\{0\}$ be a locally closed
complex submanifold of dimension $k$. Does there exist a smooth
strongly pseudoconvex complex Finsler metric $p$ on $E$ such
that, with $G=p^2$, the Levi form of $G|_{\mathcal S}$ at
every critical point is negative definite on a complex subspace
of dimension at least $k-r+1$?
\end{problem}

Our example answers Problem~\ref{prob:fulton} negatively:
there is a fixed connected $\mathcal S$ for which every
candidate metric violates the bound. This also excludes
Hermitian metrics. As in Fulton's lemma, $\mathcal S$ is a
complex-analytic submanifold and need not be algebraic.
The formulation concerns bundle metrics; it makes no assertion
about arbitrary smooth functions on $\mathcal S$.

The third problem concerns the complement of a submanifold
with ample normal bundle. One expects an exhaustion with a
prescribed number of positive Levi eigenvalues; we recall the
definition of $q$-convexity in
Section~\ref{sec:preliminaries}.
The Hartshorne--Schneider conjecture,
as stated in \cite[Conjecture~4.1, p.~241]{Demailly99}, predicts
this convexity from ampleness of the normal bundle.

\begin{conjecture}[Hartshorne--Schneider]\label{conj:hartshorne-schneider}
Let $Z$ be a smooth projective variety and let $Y\subset Z$
be a smooth subvariety of codimension $q\ge1$. If the normal
bundle $N_{Y/Z}$ is ample, then $Z\setminus Y$ is $q$-convex.
\end{conjecture}

On every Hirzebruch surface, the construction in
Theorem~\ref{thm:hirz-main} gives counterexamples to both
Problem~\ref{prob:fulton} and
Conjecture~\ref{conj:hartshorne-schneider}. The same ample stable
bundle gives a fixed connected complex-analytic hypersurface
that violates the bound in Problem~\ref{prob:fulton}, and a
submanifold with ample normal bundle whose complement in a
projective variety fails the predicted convexity. Both conclusions follow from
the analytic obstructions established in the proof of
Theorem~\ref{thm:hirz-main}.

\begin{theorem}\label{thm:applications-main}
For every integer $e\ge0$, let $B=\mathbb F_e$,
$H=S+(3e+1)f$, and $P=4H-f$. The ample $\mu_P$-stable
rank-two bundle $E$ in Theorem~\ref{thm:hirz-main}, with
$c_1(E)=4H$ and $c_2(E)=65e+26$, can be chosen so that:
\begin{enumerate}
\item\label{item:applications-fulton}
There is a connected smooth Stein hypersurface
$\mathcal S\subset E^\circ:=E\setminus\{0\}$, locally closed in
the complex-analytic topology, such that every smooth strongly
pseudoconvex squared Finsler metric $G$ on $E$ has a nondegenerate
critical point on $\mathcal S$ at which the Levi form of
$G|_{\mathcal S}$ has at least two positive eigenvalues and
at most one negative eigenvalue.
\item\label{item:applications-complement}
Put $Z=\mathbb P(\mathcal O_B\oplus E)$ and
$Y=\mathbb P(\mathcal O_B)\simeq B$. Then
$N_{Y/Z}\simeq E$ is ample, and $Z\setminus Y$ is
$3$-convex but not $2$-convex. Both $Z$ and $Y$ are rational
and simply connected. If $e\ge1$, then $Y$ is not biholomorphic
to a product of positive-dimensional compact complex manifolds.
\end{enumerate}
\end{theorem}

In assertion~(\ref{item:applications-fulton}), $\dim\mathcal S=3$ and
$\operatorname{rank}E=2$,
so Fulton's lemma would require at least two negative Levi
eigenvalues at every critical point. The fixed connected
$\mathcal S$ has a critical point violating this bound for
every candidate metric. In assertion~(\ref{item:applications-complement}),
the surface $Y$ has codimension two in the fourfold $Z$.
The conjecture would therefore require an exhaustion of $Z\setminus Y$ with at
least three positive Levi eigenvalues outside a compact set;
the theorem excludes every such exhaustion and shows that the
least smooth convexity index is three. Thus the two
counterexamples arise from the same ample stable bundle on
each simply connected surface $\mathbb F_e$.

\begin{remark}\label{rem:applications-f1}
For $e=1$, the data are $H=S+4f$, $P=4S+15f$,
$c_1(E)=4S+16f$, and $c_2(E)=91$. The base
$B\simeq Y\simeq\mathbb F_1$ is a simply connected Fano
surface and is not a product. Hence both counterexamples occur
over a non-product Fano surface.
Among Hirzebruch surfaces, $\mathbb F_1$ is the only one that
is both Fano and not a product: $\mathbb F_0$ is
$\mathbb P^1\times\mathbb P^1$, whereas for $e\ge2$ one has
$(-K_{\mathbb F_e})\cdot S=2-e\le0$; see
Remark~\ref{rem:geom-base-properties}.
\end{remark}

We adapt the strategy of Du--Xie \cite{DX}, who combine a
Hermitian integral obstruction that persists under deformation
with the density of the ample locus. For general strongly
pseudoconvex Finsler metrics, we establish an
obstruction to all smooth complex norms on the dual bundle
whose squares are plurisubharmonic. The main analytic task is
to control their nonquadratic dependence on the fibre variables
uniformly over all candidate norms. We must also pass from
Kobayashi curvature to a dual norm without assuming that the
original metric is real convex.

We average horizontal Levi forms near the line axes of a split
bundle, using a different positive test form for each summand.
The triangle inequality and integration by parts yield a strict
negative integral that persists under deformation, without
uniform bounds on the fibre Hessians of candidate norms.
Thus a single $C^4$ neighbourhood excludes all smooth complex
norms with plurisubharmonic square (Theorem~\ref{thm:an-lines}).

A smooth strongly pseudoconvex Finsler metric with semipositive
Kobayashi curvature gives a continuous dual norm whose square
is plurisubharmonic and strongly real convex, even when the
original metric is not real convex. On an ample bundle,
homogeneous regularization \cite[\S2.3, Lemma~13]{WuIII}
gives a smooth dual norm with strictly plurisubharmonic
square (Proposition~\ref{prop:an-ample-semi}), contradicting
the obstruction. Together with explicit ample and stable
comparison bundles, Walter's irreducibility theorem for
prioritary sheaves \cite[Proposition~2]{Walter} supplies ample
stable deformations whose duals lie in the obstruction
neighbourhood, proving Theorem~\ref{thm:hirz-main}.

For the smooth strongly pseudoconvex Finsler formulation of
Fulton's problem, holomorphic approximation combines local
critical-point tests into one fixed connected complex-analytic
Stein hypersurface on which every candidate metric violates
Fulton's bound.
For the Hartshorne--Schneider counterexample, take
$Z=\mathbb P(\mathcal O_B\oplus E)$ and
$Y=\mathbb P(\mathcal O_B)$. Then $Y$ has codimension two
and ample normal bundle $E$, but $Z\setminus Y$ is not
$2$-convex. Boundary geometry and regularization would turn
any such exhaustion into a smooth dual norm with
plurisubharmonic square, contrary to the obstruction.
A Hermitian--Einstein metric of positive mean curvature proves
$3$-convexity, so the least smooth convexity index is three.

Section~\ref{sec:preliminaries} fixes the metric and curvature
conventions. Section~\ref{sec:analytic} proves the uniform norm
obstruction and its persistence under deformation, and explains
how it excludes semipositive Kobayashi curvature.
Section~\ref{sec:hirzebruch-geometry} constructs the stable ample
Hirzebruch examples.
Section~\ref{sec:applications} treats the Finsler formulation in
Problem~\ref{prob:fulton} and projective complements, including
their exact convexity index.

\section{Preliminaries}\label{sec:preliminaries}

This section recalls basic facts about complex Finsler metrics,
convex norms, and Kobayashi curvature. We review the relation
between convexity and the triangle inequality, as well as
Kobayashi's characterization of ampleness. We also fix the
conventions and notation used in the arguments that follow.

\subsection{Complex Finsler metrics and convex norms}

Let $E\to X$ be a holomorphic vector bundle of rank $r$ over
a complex manifold of dimension $n$, and write
$E^\circ=E\setminus\{0\}$ for the complement of the zero section.
We use the line convention for projectivization:
$\mathbb P(E)$ parametrizes one-dimensional subspaces of the
fibres, $\mathcal O_{\mathbb P(E)}(-1)$ is the tautological
line subbundle, and $\mathcal O_{\mathbb P(E)}(1)$ is its dual.

A \emph{smooth complex Finsler metric} is a continuous function
$p:E\to[0,\infty)$ that is smooth on $E^\circ$, positive
there, and satisfies $p(\lambda v)=|\lambda|p(v)$ for
$\lambda\in\mathbb C$. We write $G=p^2$ and call it the
\emph{squared metric}; thus $G(\lambda v)=|\lambda|^2G(v)$.
Smoothness at the zero section is not required.
In local holomorphic base and fibre coordinates $(z,v)$, the
metric is \emph{strongly pseudoconvex} if the vertical complex
Hessian $(G_{i\bar j})=(\partial^2G/\partial v_i\partial\bar v_j)$
is positive definite on $E^\circ$; see
\cite{Kobayashi96,WuIII}.

We call $p$ \emph{convex} if it is convex on the underlying
real vector space of each fibre. A \emph{complex norm} is a
continuous function $p:E\to[0,\infty)$, positive away from
zero, that satisfies $p(\lambda v)=|\lambda|p(v)$ and
$p(u+v)\le p(u)+p(v)$ for vectors in the same fibre.
For a Finsler metric, fibrewise real convexity is equivalent
to the triangle inequality.
Indeed, convexity and homogeneity give
$$p(u+v)=2p\!\left(\tfrac{u+v}{2}\right)\le p(u)+p(v).$$
Conversely, the triangle inequality and homogeneity give
$$p((1-t)u+tv)\le(1-t)p(u)+tp(v)$$ for $0\le t\le1$.
We use the term \emph{smooth complex norm} when $p$ is also
smooth away from zero. This term alone does not impose strong
pseudoconvexity.

Equivalently, $p$ is a complex norm if its fibrewise unit ball
is convex, or if $G=p^2$ is fibrewise real convex.
For the second equivalence, convexity of $p$ implies convexity
of $G$ because squaring is convex and increasing on
$[0,\infty)$. Conversely, convexity of $G$ makes
$\{G\le1\}=\{p\le1\}$ convex. For nonzero $u,v$, the vector
$(u+v)/(p(u)+p(v))$ is a convex combination of $u/p(u)$ and
$v/p(v)$, and hence belongs to this unit ball. This gives
the triangle inequality.

A squared metric $G$ is \emph{strongly real convex} if its
real fibre Hessian is positive definite away from zero; we
also call $p$ a \emph{strongly convex Finsler metric} in this
case. Homogeneity makes $G$ continuously differentiable at
zero with differential zero, so the Hessian condition implies
convexity on the whole fibre. It also implies strong
pseudoconvexity, since the complex Hessian is obtained by
averaging the real Hessian in a direction and its multiple
by $i$. Strong pseudoconvexity alone does not imply real
convexity or the triangle inequality. In particular, the
term ``convex'' in \cite{Kobayashi75} refers to what we call
strongly pseudoconvex; our real-convexity terminology follows
\cite[\S2.1]{WuIII}.

For the continuous homogeneous functions used in regularization,
we adopt quantitative versions of these definitions. A continuous
function $Q$, positive away from zero and satisfying
$Q(z,\lambda v)=|\lambda|^2Q(z,v)$,
is \emph{strongly real convex} if, locally near every point of
$E^\circ$, $Q(z,v)-c|v|^2$ is convex in the real fibre
variables for some $c>0$, uniformly for nearby $z$.
It is \emph{strongly plurisubharmonic} if locally
$Q-c(|z|^2+|v|^2)$ is plurisubharmonic for some $c>0$.
These are the strict conditions used in
\cite[\S2.3, Lemma~13]{WuIII}.

\subsection{Curvature conventions and Levi forms}

We set $d^c=\frac{i}{4\pi}(\bar\partial-\partial)$ and
$dd^c=\frac{i}{2\pi}\partial\bar\partial$.
For a Hermitian line metric $k$, the real curvature form
$R_k=-dd^c\log k$ represents $c_1$.
More generally, $R_h$ denotes the Chern curvature of a Hermitian
vector bundle multiplied by $i/(2\pi)$.
When a real $(1,1)$-form or such a curvature is evaluated on
$(u,\bar u)$, we mean its associated Hermitian form.

For a smooth real function $u$, its \emph{Levi form}
$\mathcal L_u$ is the Hermitian form with matrix
$(u_{A\bar B})$ in local holomorphic coordinates.
The function is plurisubharmonic if this form is semipositive,
and strictly plurisubharmonic if it is positive definite.
For continuous functions, plurisubharmonicity means that the
pullback to every holomorphic disc is subharmonic.
The \emph{negative Levi index} is the maximal complex dimension
of a subspace on which $\mathcal L_u$ is negative definite.
In particular, it counts negative complex Levi eigenvalues,
rather than negative eigenvalues of the real Hessian.

For a strongly pseudoconvex squared Finsler metric $G$, the
\emph{Kobayashi curvature} is represented by the horizontal form
\begin{equation}\label{eq:conv-kobayashi}
 \mathcal K_G=
 -\frac{i}{2\pi G}
 \left(G_{\alpha\bar\beta}
       -G_{\alpha\bar j}G^{\bar j i}G_{i\bar\beta}\right)
 dz^\alpha\wedge d\bar z^\beta,
\end{equation}
where $(G^{\bar j i})$ is the inverse vertical Hessian and
repeated indices are summed. Greek indices refer to base
coordinates and Latin indices to fibre coordinates.
Positive, semipositive, and negative Kobayashi curvature mean
that this form is respectively positive definite, semipositive,
and negative definite in the base directions at every nonzero
vector. For a Hermitian squared norm $G(v)=h(v,v)$, these
notions coincide with Griffiths positivity, semipositivity,
and negativity; see \cite{Kobayashi75,Kobayashi96}.

Let $V_v$ be the vertical tangent space at $v\ne0$, and let
$H_{G,v}$ be its Levi-orthogonal complement with respect to
$\mathcal L_G$. Strong pseudoconvexity gives the splitting
$T_v^{1,0}E=H_{G,v}\oplus V_v$.
For the radial vector
$R=\sum_i v_i\partial_{v_i}$, complex homogeneity gives
\begin{equation}\label{eq:conv-euler}
 \partial G(R)=G,\qquad
 \mathcal L_G(a,\bar R)=\partial G(a),\qquad
 \partial G|_{H_{G,v}}=0.
\end{equation}
The horizontal restriction of $\mathcal L_G$ is the Schur
complement in \eqref{eq:conv-kobayashi}; it is nonpositive when
Kobayashi curvature is semipositive, and negative definite
when Kobayashi curvature is positive.

\subsection{Kobayashi's characterization of ampleness}

On a compact base, a holomorphic vector bundle $E$ is
\emph{ample} if $\mathcal O_{\mathbb P(E^*)}(1)$ is ample
\cite{Hartshorne66}. The following characterization explains
the role of Finsler metrics on the dual bundle.

\begin{theorem}[Kobayashi {\cite[Theorem~5.1]{Kobayashi75}}]
\label{thm:prelim-kobayashi}
Let $E$ be a holomorphic vector bundle on a compact complex
manifold. The following conditions are equivalent:
\begin{enumerate}
\item $E$ is ample.
\item $E^*$ admits a smooth strongly pseudoconvex complex
Finsler metric with negative Kobayashi curvature.
\item $E^*$ admits a smooth squared Finsler metric $G$ that
is strictly plurisubharmonic on $(E^*)^\circ$.
\end{enumerate}
\end{theorem}

Kobayashi proves the equivalence of the first two conditions.
The equivalence of the last two conditions follows directly from
\eqref{eq:conv-kobayashi}: with a positive definite vertical
block, the full Levi matrix is positive definite exactly when
its horizontal Schur complement is positive definite, which
is equivalent to negative Kobayashi curvature. The characterization asserts the existence of a strongly
pseudoconvex metric; it does not assert that this metric is
real convex or satisfies the triangle inequality.

\subsection{Convexity of complements and other notation}

A complex manifold of dimension $N$ is \emph{$q$-convex}
if it admits a smooth exhaustion $\psi$ whose Levi form has
at least $N-q+1$ positive eigenvalues outside a compact set.
Here an exhaustion is a function for which each sublevel set
$\{\psi\le c\}$ is compact. Throughout, $q$-convexity
refers to smooth exhaustions outside a compact set; it is not
the stronger requirement of $q$-completeness, which imposes
the same Levi condition everywhere. The least smooth convexity
index is the smallest integer $q$ for which such a smooth
exhaustion exists.
This is the analytic notion used in the
Hartshorne--Schneider conjecture
\cite[\S4]{Demailly99}.

For an ample divisor $P$ on a smooth projective $n$-fold and
a torsion-free coherent sheaf $\mathcal F$ of positive rank,
we write
$\mu_P(\mathcal F)=c_1(\mathcal F)\cdot P^{n-1}/\rk\mathcal F$.
The sheaf is \emph{$\mu_P$-stable} if every coherent subsheaf
$0\ne\mathcal F'\subset\mathcal F$ of smaller positive rank
satisfies $\mu_P(\mathcal F')<\mu_P(\mathcal F)$
\cite[Definitions~1.2.11--1.2.12]{HuybrechtsLehn}.
On a surface, the degree of a line bundle $L$ is
$\deg_P L=c_1(L)\cdot P$, and $c_2$ is identified with its
degree. We write $\mathcal F(D)=\mathcal F\otimes\mathcal O_X(D)$.

All $C^m$ norms on a fixed compact smooth manifold or bundle
are measured using fixed smooth background metrics and
connections. For a smooth tensor $A$, one may take
$\|A\|_{C^m}=\sum_{j=0}^m\sup|\nabla^j A|$; different
background choices give equivalent norms.
When the base is compact, the $C^m$ norms of homogeneous
functions are taken on a fixed compact unit sphere bundle.

\section{A uniform obstruction for convex plurisubharmonic norms}
\label{sec:analytic}

This section establishes the analytic obstruction used in our
construction. For a direct sum of line bundles satisfying suitable
positive-degree conditions, we derive an integral inequality that
excludes smooth complex norms with plurisubharmonic square.
The estimate is uniform over all candidate norms and persists
under small deformations of the bundle's holomorphic structure
on the fixed base. We then use duality and regularization to
show that such an obstruction on the dual bundle $E^*$ rules out smooth strongly
pseudoconvex Finsler metrics with semipositive Kobayashi curvature
on an ample bundle $E$.

\subsection{Positive test forms and curvature normalization}

Let $X$ be a connected compact complex manifold of dimension
$n\ge2$. A real $(n-1,n-1)$-form $\Theta$ is \emph{strictly positive}
if, pointwise, its wedge product with every nonzero semipositive
$(1,1)$-form is a positive top-degree form. Assume $dd^c\Theta=0$. For a
holomorphic line bundle $F$, set
\begin{equation}\label{eq:an-degree}
 \deg_\Theta F=\int_XR_k\wedge\Theta.
\end{equation}
This is independent of $k$, since integration by parts gives
$\int_Xdd^cu\wedge\Theta=\int_Xu\,dd^c\Theta=0$.
It is the pairing of the Bott--Chern first Chern class with the
Aeppli class of $\Theta$. It need not be a de Rham intersection
number, since $\Theta$ need not be $d$-closed.
A Gauduchon metric $\omega$ gives such a form
$\Theta=\omega^{n-1}$.

\begin{lemma}\label{lem:an-normalize}
Fix a smooth positive volume form $\mu$ on $X$. For every
Hermitian line metric $k^0$, there is a smooth real function $u$
such that $k=k^0e^{-u}$ satisfies
\begin{equation}\label{eq:an-normalize}
 R_k\wedge\Theta
 =\frac{\deg_\Theta F}{\int_X\mu}\,\mu.
\end{equation}
\end{lemma}

\begin{proof}
Define the real scalar operator
$\mathcal Lu={dd^cu\wedge\Theta}/{\mu}.$
Strict positivity of $\Theta$ makes this a uniformly elliptic
second-order operator. It has no zero-order term; its kernel
consists of constants by the strong maximum principle and
connectedness of $X$. Acting from $C^{k+2,\alpha}(X)$ to
$C^{k,\alpha}(X)$, for $0<\alpha<1$, it is Fredholm of index
zero. Indeed, interpolate it with a scalar Laplacian having the
same sign convention. The principal symbols remain elliptic
throughout this interpolation, so the index is the index of the
Laplacian. Its cokernel is therefore one-dimensional. On the
other hand,
$\int_X\mathcal Lu\,\mu =\int_Xdd^cu\wedge\Theta=0.$
Thus its image consists exactly of the functions with $\mu$-mean zero.
The right-hand side of
$$dd^cu\wedge\Theta =\frac{\deg_\Theta F}{\int_X\mu}\,\mu -R_{k^0}\wedge\Theta$$
has integral zero. Solve this equation and use elliptic regularity
to obtain a smooth real $u$. Since $R_{k^0e^{-u}}=R_{k^0}+dd^cu$,
this proves \eqref{eq:an-normalize}.
\end{proof}

The same argument applies to a real closed $(1,1)$-form
$\theta$: one can add $dd^cu$ so that
$(\theta+dd^cu)\wedge\Theta$ is a constant multiple of $\mu$.

\subsection{Direct sums of line bundles}

\begin{theorem}\label{thm:an-lines}
Let $F_0=\bigoplus_{s=1}^rF_s$ be a direct sum of holomorphic
line bundles on $X$, where $r\ge2$. Suppose that there are smooth
strictly positive real $(n-1,n-1)$-forms $\Theta_s$ such that
\begin{equation}\label{eq:an-hypothesis}
 dd^c\Theta_s=0,
 \qquad \deg_{\Theta_s}F_s>0
 \qquad (1\le s\le r).
\end{equation}
There is a $C^4$ neighbourhood of the split Dolbeault operator
on the underlying smooth complex vector bundle in which no
integrable operator admits a smooth complex norm with
plurisubharmonic square away from the zero section.
\end{theorem}

Here a neighbouring operator is written
$\bar\partial_A=\bar\partial_0+A$, with
$A\in\Omega^{0,1}(X,\operatorname{End}F_0)$ and
$\bar\partial_A^2=0$. Its $C^4$ size is measured using any fixed
smooth metrics and connections. On the compact base all such
choices give equivalent norms.

\subsubsection*{The fixed geometric data.}
Fix a smooth positive volume form $\mu$. By
Lemma~\ref{lem:an-normalize}, choose metrics $k_s$ on $F_s$ with
curvatures satisfying
\begin{equation}\label{eq:an-curvatures}
 R_s\wedge\Theta_s=\lambda_s\mu,
 \qquad
 \lambda_s=\frac{\deg_{\Theta_s}F_s}{\int_X\mu}>0.
\end{equation}
Let $k=\bigoplus_sk_s$, let $P_s$ be the $k$-orthogonal
projections, and let
$\Sigma=\{v\in F_0:|v|_k=1\}.$
We use the round measure $d\sigma_x$ of mass one on each fibre
sphere. Together with $\mu$ it defines a fixed smooth positive
density $d\nu=d\sigma_x\,\mu(x)$ on the compact manifold
$\Sigma$.

Choose a nonnegative smooth function $\chi$ supported in
$(-1,1)$ and positive near zero. For $0<\varepsilon<1/2$, we define the following function on $\Sigma$:
\begin{equation}\label{eq:an-densities}
 \rho_s(v)=c_\varepsilon\,
 \chi\left(\tfrac{|(I-P_s)v|_k^2}{\varepsilon^2}\right),
 \qquad
 \int_{\Sigma_x}\rho_s\,d\sigma_x=1.
\end{equation}

In a local $k$-unitary frame $e_1,\ldots,e_r$ respecting the splitting $F_0=\bigoplus_jF_j$, write $v=\sum_jz_je_j$. Each fibre sphere is then identified with the standard sphere $S^{2r-1}$, and
$\rho_s(z) =c_\varepsilon\, \chi({\sum_{j\ne s}|z_j|^2}/{\varepsilon^2}).$
This coordinate expression does not depend on the base point $x$. Moreover, permuting the coordinates preserves the spherical measure and exchanges the expressions for different indices $s$. Hence the same constant $c_\varepsilon$, independent of both $x$ and $s$, normalizes each $\rho_s$ to have integral one.

The functions $\rho_s$ are globally defined, smooth, and nonnegative. They are invariant under independent phase rotations $z_j\mapsto e^{i\theta_j}z_j$, so their expressions do not depend on the choice of unitary frame in each line summand. Thus $\rho_s\,d\sigma_x$ defines a probability density on each fibre sphere.

The support of $\rho_s$ is contained in
$\{v\in\Sigma: |(I-P_s)v|_k<\varepsilon\},$
so it is concentrated near the unit circle of the summand $F_s$. We will first choose $\varepsilon$ sufficiently small for the estimate below, and then keep $\varepsilon$ and the resulting weights fixed while deforming the Dolbeault operator.

\subsubsection*{Estimates supplied by convexity.}
For a smooth complex norm $p$, let $G=p^2$ and define
$$a_s(x)=G(e_s),\qquad N(x)=\sum_{s=1}^ra_s(x),$$
where $e_s$ is any $k$-unit vector in $(F_s)_x$. Complex homogeneity makes $a_s$ independent of the choice of phase. The triangle
inequality and Cauchy--Schwarz give
\begin{equation}\label{eq:an-C1}
 G(v)\le N|v|_k^2,
 \qquad |d_vG(w)|\le2N|v|_k|w|_k
\end{equation}
for $v,w\in(F_0)_x$, with $v\ne0$ in the differential estimate.
Indeed, $p(v)\le\sum_s|v_s|\sqrt{a_s}\le\sqrt N|v|_k$,
and differentiating $|p(v+tw)-p(v)|\le|t|p(w)$ for real $t$
gives $|d_vp(w)|\le p(w)$.

For each $j$, define
$D_jG(v)=(d_vG)(P_jv) =\left.\frac{d}{dt}\right|_{t=0}G(v+tP_jv).$
This measures the change in $G$ when only the $j$-th component of $v$ is rescaled. We claim that, on the support of $\rho_s$,
\begin{equation}\label{eq:an-axis}
|D_jG|\le2\varepsilon N\quad(j\ne s),
\qquad
|D_sG-2a_s|\le10\varepsilon N.
\end{equation}

Fix $x$ and $v\in\Sigma_x\cap\operatorname{supp}\rho_s$, and set $w=(I-P_s)v$. Then $|v|_k=1$ and $|w|_k<\varepsilon$. For $j\ne s$, orthogonality gives $|P_jv|_k\le|w|_k$, so the derivative estimate yields
$|D_jG(v)| \le2N|v|_k|P_jv|_k \le2\varepsilon N.$

To estimate $D_sG$, let
$\ell=|P_sv|_k,$ $e_s=\frac{P_sv}{\ell}.$
Here $\ell>0$, since $\ell^2=1-|w|_k^2$. Moreover, $G(e_s)=a_s$, and the reverse triangle inequality gives $1-\ell\le|w|_k$. Consequently,
$|v-e_s|_k \le|w|_k+(1-\ell) <2\varepsilon.$
The segment $\gamma(t)=(1-t)e_s+tv$ lies in the closed unit ball. It also avoids zero, since every point on it is within distance $2\varepsilon<1$ of the unit vector $e_s$. Integrating the derivative estimate along this segment therefore gives
$$|G(v)-a_s| \le\int_0^1 2N|\gamma(t)|_k|v-e_s|_k\,dt \le4\varepsilon N.$$

Finally, linearity of the differential gives
$|\sum_{j\ne s}D_jG| =|(d_vG)(w)| \le2\varepsilon N.$
Since $G$ is homogeneous of degree two, Euler's identity implies
$D_sG-2a_s =2(G(v)-a_s)-\sum_{j\ne s}D_jG.$
Hence
$$|D_sG-2a_s| \le2(4\varepsilon N)+2\varepsilon N =10\varepsilon N,$$
as claimed.

\subsubsection*{The averaged horizontal Levi identity.}

Let $\pi:F_0\to X$ be the bundle projection, and let $D_0$ be the Chern connection of the split Hermitian metric $k$. For $v\in\Sigma_x$ and a real tangent vector $\xi\in T_xX$, choose a base curve $\gamma(t)$ with $\gamma(0)=x$ and $\dot\gamma(0)=\xi$, and let $v(t)$ be the $D_0$-parallel transport of $v$ along $\gamma$. The real horizontal lift is defined by
$$H_{0,v}(\xi)=\dot v(0), \qquad d\pi_v\bigl(H_{0,v}(\xi)\bigr)=\xi.$$

This lift is compatible with the complex structures on the base and the total space:
$H_{0,v}(J_X\xi)=J_{F_0}H_{0,v}(\xi).$
Indeed, in a local holomorphic frame, the connection matrix of $D_0$ has type $(1,0)$, which makes its horizontal subspaces invariant under $J_{F_0}$.

Moreover, $D_0$-parallel transport preserves the $k$-norm. Since $v$ has unit norm, the lifted curve $v(t)$ remains in the unit sphere bundle $\Sigma$. Its velocity therefore satisfies
$H_{0,v}(\xi)\in T_v\Sigma.$
Thus $H_{0,v}$ is a complex-linear map into $T_vF_0$ whose image is tangent to $\Sigma$.

 Define
$\beta_s(G)_x= \int_{\Sigma_x}\rho_s(v)H_{0,v}^*(dd^cG)_v\,d\sigma_x(v),$
which is a real $(1,1)$-form on $X$. Equivalently, for any $\xi,\eta\in T_xX$,
\[
\beta_s(G)_x(\xi, \eta)=\int_{\Sigma_x} \rho_s(v)\left(d d^c G\right)_v\left(H_{0, v} \xi, H_{0, v} \eta\right) d \sigma_x(v).
\]
Put
$U_s=\int_{\Sigma_x}G\rho_s\,d\sigma_x,$ $V_{sj}=\int_{\Sigma_x}D_jG\rho_s\,d\sigma_x.$
For every smooth real function $G$ on $F_0\setminus\{0\}$,
\begin{equation}\label{eq:an-average}
 \beta_s(G)=dd^cU_s-\frac12\sum_jR_jV_{sj}.
\end{equation}

Fix a point $x_0\in X$. Choose local holomorphic coordinates $z$ near $x_0$ and holomorphic frames $e_j$ of the line summands such that
$k_j(e_j,e_j)=e^{-\varphi_j},$ $d\varphi_j(x_0)=0.$
In these frames the Chern connection forms are $-\partial\varphi_j$, so they vanish at $x_0$. Consequently, the horizontal lifts at $x_0$ have no vertical component, and the coefficients of $H_{0,v}^*(dd^cG)$ are determined by the base derivatives $G_{\alpha\bar\beta}$.

Introduce the smooth unitary fibre coordinates
$\zeta_j=e^{-\varphi_j(z)/2}v_j,$
and write $G(z,v)=g(z,\zeta)$. These coordinates identify every fibre sphere with the fixed standard sphere $\sum_j|\zeta_j|^2=1$. In what follows, derivatives of $G$ with respect to $z$ are taken with $v$ fixed, whereas derivatives of $g$ are taken with $\zeta$ fixed.
At $x_0$, the first base derivatives of $\zeta_j$ and $\bar\zeta_j$ vanish, while
$\partial_\alpha\partial_{\bar\beta}\zeta_j =-\frac12(\varphi_j)_{\alpha\bar\beta}\zeta_j,$ $\partial_\alpha\partial_{\bar\beta}\bar\zeta_j =-\frac12(\varphi_j)_{\alpha\bar\beta}\bar\zeta_j.$
Thus all terms in the mixed chain rule involving first derivatives of the coordinate change vanish at $x_0$, leaving
$$G_{\alpha\bar\beta} = g_{\alpha\bar\beta} -\frac12\sum_j(\varphi_j)_{\alpha\bar\beta} ( \zeta_jg_{\zeta_j} +\bar\zeta_jg_{\bar\zeta_j} ).$$
The expression in parentheses is $D_j G=\left.\frac{d}{d t}\right|_{t=0} G\left(v+t P_j v\right)$, since rescaling the $j$-th component of $v$ also rescales the $j$-th component of $\zeta$.

In unitary coordinates, both the standard spherical measure $d\sigma$ and the weight $\rho_s(\zeta)$ are independent of $z$. Hence
$U_s(z)=\int_{S^{2r-1}}g(z,\zeta)\rho_s(\zeta)\,d\sigma(\zeta),$
and differentiation under the integral gives
$\int_{S^{2r-1}}g_{\alpha\bar\beta}\rho_s\,d\sigma =(U_s)_{\alpha\bar\beta}.$
Averaging the chain-rule identity therefore yields
$$\int_{\Sigma_{x_0}}G_{\alpha\bar\beta}\rho_s\,d\sigma_{x_0} = (U_s)_{\alpha\bar\beta} -\frac12\sum_j(\varphi_j)_{\alpha\bar\beta}V_{sj}$$
at $x_0$. Using $dd^c=\frac{i}{2\pi}\partial\bar\partial$ and $R_j=dd^c\varphi_j$, we obtain
\eqref{eq:an-average}.
Since $x_0$ was arbitrary, the identity holds on all of $X$.

\subsubsection*{A strict inequality uniform over all norms.}
Set
$I_0(G)=\sum_s\int_X\beta_s(G)\wedge\Theta_s.$
It is nonnegative when $G$ is plurisubharmonic. However,
$dd^c\Theta_s=0$ and integration by parts give
$\int_Xdd^cU_s\wedge\Theta_s=0$. Define
\[
 M_{j,s}=\sup_X\left|
       \frac{R_j\wedge\Theta_s}{\mu}\right|,
 \quad
 C=5\sum_s\lambda_s+\sum_s\sum_{j\ne s}M_{j,s},
 \quad \lambda_{\min}=\min_s\lambda_s.
\]
Using \eqref{eq:an-curvatures}, \eqref{eq:an-axis}, and
\eqref{eq:an-average}, we obtain
\begin{align}
 I_0(G)
 &=-\frac12\sum_{s,j}\int_XV_{sj}R_j\wedge\Theta_s\notag\\
 &\le-\int_X\sum_s\lambda_sa_s\,\mu
          +\varepsilon C\int_XN\,\mu\notag\\
 &\le-(\lambda_{\min}-\varepsilon C)\int_XN\,\mu.
 \label{eq:an-margin}
\end{align}
Fix $\varepsilon$ such that $\varepsilon C<\lambda_{\min}/2$.
This gives a negative bound independent of all derivatives of
$G$ in base directions and of its second derivatives in fibre
directions.

\subsection{Persistence via a differential operator on the sphere}
\label{subsec:an-persistence}

We now prove that the strict inequality persists under small
changes in the bundle's Dolbeault operator on the fixed complex
base $X$.
For the fixed metric $k$, the Chern connection associated with
$\bar\partial_A=\bar\partial_0+A$ is
$D_A=D_0+A-A^{*k}.$
Let $J_A$ be the resulting total-space complex structure and
$H_{A,v}$ the real horizontal lift. These depend smoothly and
algebraically on $A$ in fixed smooth coordinates. The lift is
complex linear and preserves the unit sphere bundle.

For an arbitrary real function $g\in C^\infty(\Sigma)$, define its
radial extension
$$(\mathcal Eg)(v)=|v|_k^2g\left(v/|v|_k\right),\qquad v\ne0.$$
This extension is independent of the holomorphic structure. For
any degree-two homogeneous function $G$, it gives
$\mathcal E(G|_\Sigma)=G$. Define a global real differential
operator on $\Sigma$ by
\begin{equation}\label{eq:an-operator}
 (\mathcal P_Ag)(v)=\sum_s\rho_s(v)
 \frac{H_{A,v}^*(dd^c_{J_A}\mathcal Eg)_v
                  \wedge\Theta_s(x)}{\mu(x)},
 \qquad x=\pi(v).
\end{equation}
It has order at most two. Indeed, differentiating $\mathcal Eg$
twice and restricting to the unit sphere involves at most the
second derivatives of $g$ there. The real identity
$$d^c_Jf=-\frac1{4\pi}\,df\circ J, \qquad dd^c_Jf=-\frac1{4\pi}\,d(df\circ J)$$
shows that the coefficients of \eqref{eq:an-operator} depend
on $A$ and its first derivatives, together with fixed smooth
data. Consequently they converge in $C^2$ when $A\to0$ in
$C^4$.

Take the formal adjoint with respect to $d\nu$ and set $b_A=\mathcal P_A^*1$.
If, in a coordinate chart, $d\nu=m(y)\,dy$ and
$\mathcal P_A=\sum_{|\gamma|\le2}c_{\gamma,A}\partial^\gamma$,
then
\begin{equation}\label{eq:an-adjoint}
 b_A=m^{-1}\sum_{|\gamma|\le2}(-1)^{|\gamma|}
                    \partial^\gamma(mc_{\gamma,A}).
\end{equation}
This gives a globally defined smooth function, and
\begin{equation}\label{eq:an-adjoint-continuity}
 \|b_A-b_0\|_{C^0(\Sigma)}\longrightarrow0
 \quad\text{as }\|A\|_{C^4}\longrightarrow0.
\end{equation}
For a squared norm, write $g=G|_\Sigma$. The corresponding
horizontal Levi functional is
$I_A(G)=\int_\Sigma\mathcal P_Ag\,d\nu =\int_\Sigma g b_A\,d\nu.$
Since $0<g(v)\le N(\pi(v))$,
\begin{equation}\label{eq:an-uniform-perturbation}
 |I_A(G)-I_0(G)|
 \le\|b_A-b_0\|_{C^0(\Sigma)}\int_XN\,\mu.
\end{equation}
This estimate requires no bounds on derivatives of the candidate
norm.

\begin{proof}[Proof of Theorem~\ref{thm:an-lines}]
Choose the fixed data and $\varepsilon$ as above. By
\eqref{eq:an-adjoint-continuity}, for all sufficiently small
$A$ the coefficient on the right of
\eqref{eq:an-uniform-perturbation} is less than
$\lambda_{\min}/4$. Combining this with
\eqref{eq:an-margin} gives
$$I_A(G)\le-\frac{\lambda_{\min}}4\int_XN\,\mu<0$$
for every smooth squared complex norm. If $G$ were
plurisubharmonic for $J_A$, the pullbacks
$H_{A,v}^*dd^c_{J_A}G$ would be semipositive $(1,1)$-forms.
Their wedge products with the positive $\Theta_s$, averaged against
$\rho_s\ge0$, would give $I_A(G)\ge0$. This contradiction
proves the theorem.
\end{proof}

\begin{remark}\label{rem:an-order}
The angular width is fixed before the neighbourhood of the
Dolbeault operator is chosen. Derivatives of the angular density
may grow as its width decreases, but they are fixed smooth
coefficients in \eqref{eq:an-adjoint}. Neither this proof nor the
strict estimate requires uniform bounds on fibre Hessians of
norms. If all testing forms were equal, all summands of $F_0$
would have positive degree for that one form, and $F_0^*$ could
not have ample determinant. The use of different testing forms
is therefore essential in the applications below.
\end{remark}

\subsection{Passage to curvature obstructions}

We show that, on an ample bundle, semipositive Kobayashi
curvature produces a smooth dual norm with plurisubharmonic square.
We first isolate the regularization step, which will also be
used for support norms arising from concave neighbourhoods.

\begin{lemma}\label{lem:an-regularization}
Let $E$ be ample over a compact complex manifold. Suppose that
$E^*$ admits a continuous squared complex norm $Q$ which is
plurisubharmonic away from zero and such that
$Q-c|\xi|_k^2$ is fibrewise real convex for some Hermitian
metric $k$ on $E^*$ and some $c>0$. Then $E^*$ admits a
smooth strongly real convex complex norm whose square is
strictly plurisubharmonic away from zero.
\end{lemma}

\begin{proof}
By Theorem~\ref{thm:prelim-kobayashi}, there is a smooth
squared Finsler metric $H$ on $E^*$ that is strictly
plurisubharmonic off zero. Put $G_0(\xi)=|\xi|_k^2$.
Compactness of the unit sphere and degree-two homogeneity
give a constant $C\ge0$ such that $H+CG_0$ is fibrewise real convex.
Indeed, its real fibre Hessian is nonnegative off zero, and
the function is $C^1$ at zero with differential zero, so
convexity extends along every real line through zero.

Choose $t>0$ with $tC<c/2$ and put
$\widetilde Q=Q+tH$. This function is strongly
plurisubharmonic off zero, and
\begin{equation}\label{eq:an-strict-regularization}
 \widetilde Q-\frac c2G_0
 =(Q-cG_0)+t(H+CG_0)+(\frac c2-tC)G_0
\end{equation}
is fibrewise convex. The ratio $\widetilde Q/G_0$ descends
to a positive continuous function on $\mathbb P(E^*)$.
The homogeneous regularization theorem
\cite[\S2.3, Lemma~13]{WuIII} gives a smooth positive
degree-two homogeneous function $\widehat Q$ retaining both
strong real convexity and strict plurisubharmonicity.
It extends continuously by zero on the zero section.
Hence, $\sqrt{\hat{Q}}$ is the required smooth complex norm.
\end{proof}

We next apply this lemma to the polar of a semipositive metric.
\begin{proposition}
\label{prop:an-ample-semi}
Let $E$ be an ample holomorphic vector bundle on a compact
complex manifold. If $E$ admits a smooth strongly pseudoconvex
Finsler metric with semipositive Kobayashi curvature, then $E^*$
admits a smooth complex norm whose square is strongly real
convex in each fibre and strictly plurisubharmonic away from zero.
\end{proposition}

\begin{proof}
Let $X$ be the base, with $n=\dim_{\mathbb C}X$, and let $G$
be the given squared metric. Consider the auxiliary
polar function
\begin{equation}\label{eq:an-polar}
Q:E^*\longrightarrow [0,\infty),\quad Q(x,\xi)=\max_{[v]\in\mathbb P(E_x)}
                 \frac{|\xi(v)|^2}{G(x,v)}.
\end{equation}
It is continuous, positive away from zero, and satisfies
$Q(x,\lambda\xi)=|\lambda|^2Q(x,\xi)$; its square root is a
complex norm. We do not assume that $Q$ is smooth.

\textbf{Plurisubharmonicity.}
On a holomorphic coordinate ball with coordinates $z$, set
$G_\epsilon=e^{-\epsilon|z|^2}G$, where $\epsilon>0$.
Its Kobayashi curvature is
$\mathcal K_{G_\epsilon}=\mathcal K_G+ \epsilon dd^c|z|^2>0$, and its polar square is
$Q_\epsilon=e^{\epsilon|z|^2}Q$.
At any $v_0\ne0$, \eqref{eq:conv-euler} and strict positivity
of curvature give an $n$-dimensional Levi-negative horizontal
space contained in $\ker\partial G_\epsilon(v_0)$.
Choose a local holomorphic section $\sigma_0$ through $v_0$
whose tangent space at $x_0$ is $H_{G_\epsilon,v_0}$.
Since
$H_{G_\epsilon,v_0}\subset \ker\partial G_\epsilon(v_0)$,
the function $G_\epsilon\circ\sigma_0$ has vanishing first
derivative at $x_0$, and its Levi form there is negative
definite.

Its full quadratic Taylor term may, however, also contain pure
holomorphic terms $P(z)$ and their complex conjugates. These
terms can be removed by changing only the quadratic jet of the
section: Euler's identity gives $\partial_vG_\epsilon(v_0)\ne0$,
so choose a vertical vector $b$ with
$\partial_vG_\epsilon(v_0)(b)=1$ and subtract $bP(z)$ from
$\sigma_0(z)$. This cancels $P(z)+\overline{P(z)}$ without
changing the mixed quadratic term. Hence
\[
G_\epsilon(z,\sigma(z))
=
G_\epsilon(x_0,v_0)
+
\sum_{\alpha,\beta}
H_{\alpha\bar\beta}z_\alpha\bar z_\beta
+O(|z|^3),
\]
where $(H_{\alpha\bar\beta})$ is negative definite.
After shrinking the coordinate neighbourhood, we obtain
$$G_\epsilon(z,\sigma(z)) \le G_\epsilon(x_0,v_0),\quad \sigma(x_0)=v_0.$$
This is the local osculating-section argument in
\cite[Proposition~2.4]{Demailly99}.
Choose $v_0$ on a maximizing line for $Q_\epsilon(x_0,\xi_0)$,
where $\xi_0\ne0$. Then
$$\log Q_\epsilon(z,\xi)\ge \log|\xi(\sigma(z))|^2-\log G_\epsilon(x_0,v_0),$$
with equality at $(x_0,\xi_0)$. The right-hand side is pluriharmonic
near that point, since the pairing is nonzero there. The submean
inequality on every sufficiently small complex disc proves that
$\log Q_\epsilon$ is plurisubharmonic. As $\epsilon\downarrow0$,
the submean inequality passes to the locally uniform limit.
Thus $\log Q$, and hence $Q$, is plurisubharmonic away from zero.
The continuous extension of $Q$, equal to zero on the zero
section, is plurisubharmonic there as well by the removable
singularity theorem \cite[Chapter~I, Theorem~5.24]{Demailly}.

\textbf{Uniform real strong convexity.}
Fix a smooth Hermitian metric $k$ on $E$. Smoothness of $G$
away from zero and degree-two homogeneity give a constant $L>0$
such that, in every fibre,
\begin{equation}\label{eq:an-upper-quadratic}
 G(v+w)\le G(v)+d_vG(w)+L|w|_k^2.
\end{equation}
Indeed, the real fibre Hessian is bounded above on the compact
$k$-unit sphere bundle. Homogeneity extends the bound to all
nonzero vectors, and $G$ is $C^1$ at zero with differential zero.
Thus the estimate also holds for segments passing through zero.

 Fix $x\in X$ and
$\xi\in E_x^*$. 
Using the fibrewise
Legendre--Fenchel representation, we can write
\begin{equation}\label{eq:an-legendre}
 Q(x,\xi)
 =
 \max_{v\in E_x}
 \bigl\{2\operatorname{Re}\xi(v)-G(x,v)\bigr\}.
\end{equation}
This is the usual Legendre--Fenchel expression for the real
pairing
$(\xi,v)\to2\operatorname{Re}\xi(v);$
see, for example, \cite[\S12]{Rockafellar70}.
The identity above follows directly from homogeneity
and completing the square, and does not require $G$ to be
real convex.

Now fix $\xi_0\in E_x^*$ and choose a maximizer $v_0$ in
\eqref{eq:an-legendre}. Such a maximizer exists. 
At the maximizer $v_0$, differentiation in the real fibre
variables gives
\begin{equation}\label{eq:an-maximizer-equation}
 dG(v_0)=2\operatorname{Re}\xi_0.
\end{equation}

Let $\zeta\in E_x^*$. To estimate $Q(\xi_0+\zeta)$, evaluate
\eqref{eq:an-legendre} at $v_0+w$, where $w\in E_x$ is arbitrary.
Using \eqref{eq:an-upper-quadratic}, we obtain
\[
\begin{aligned}
Q(\xi_0+\zeta)
&\ge
2\operatorname{Re}(\xi_0+\zeta)(v_0+w)
       -G(v_0+w)\\
&\ge
2\operatorname{Re}(\xi_0+\zeta)(v_0+w)
       -G(v_0)-dG(v_0)(w)-L|w|_k^2.
\end{aligned}
\]
Since $v_0$ maximizes \eqref{eq:an-legendre},
$Q(\xi_0) = 2\operatorname{Re}\xi_0(v_0)-G(v_0),$
and \eqref{eq:an-maximizer-equation} cancels the terms involving
$\xi_0(w)$. Thus
$Q(\xi_0+\zeta) \ge Q(\xi_0)+2\operatorname{Re}\zeta(v_0) +2\operatorname{Re}\zeta(w)-L|w|_k^2.$
Since this holds for every $w\in E_x$, we may maximize in $w$.
Completing the square once more gives
$\sup_{w\in E_x} \bigl\{ 2\operatorname{Re}\zeta(w)-L|w|_k^2 \bigr\} = L^{-1}|\zeta|_{k^*}^2.$
Consequently,
\begin{equation}\label{eq:an-strong-support}
Q(\xi_0+\zeta)
\ge
Q(\xi_0)+2\operatorname{Re}\zeta(v_0)
       +L^{-1}|\zeta|_{k^*}^2.
\end{equation}

Set $c=L^{-1}$. Subtracting
$c|\xi_0+\zeta|_{k^*}^2$ from both sides of
\eqref{eq:an-strong-support} yields
\[
\begin{aligned}
(Q-c|\cdot|_{k^*}^2)(\xi_0+\zeta)
\ge{}&
(Q-c|\cdot|_{k^*}^2)(\xi_0)+
2\operatorname{Re}\zeta(v_0)
-2c\operatorname{Re}
   \langle\xi_0,\zeta\rangle_{k^*}.
\end{aligned}
\]
The right-hand side is affine in $\zeta$ and agrees with
$Q-c|\cdot|_{k^*}^2$ at $\zeta=0$. To see why this
implies convexity, fix the fibre $E_x^*$ and put
$\Phi(\xi)=Q(x,\xi)-c|\xi|_{k^*}^2$.
For every $\eta\in E_x^*$, the preceding inequality provides
a real affine function $\ell_\eta$ such that
$\ell_\eta\le\Phi$ on $E_x^*$ and $\ell_\eta(\eta)=\Phi(\eta)$.
Consequently, $\Phi(\xi)=\sup_{\eta\in E_x^*}\ell_\eta(\xi)$,
so $\Phi$ is convex as a pointwise supremum of real affine
functions. Thus
\begin{equation}\label{eq:an-polar-strong-convex}
 Q-c|\xi|_{k^*}^2
\text{ is fibrewise real convex},
 \quad c=L^{-1}>0.
\end{equation}
In particular, $Q$ is uniformly strongly real convex in the
fibres. Neither real convexity of $G$ nor uniqueness of the
maximizer $v_0$ is needed.

\textbf{A smooth replacement.}
The function $Q$ is a continuous plurisubharmonic squared norm,
and \eqref{eq:an-polar-strong-convex} gives its uniform strong
real convexity. Lemma~\ref{lem:an-regularization} therefore
provides the required smooth norm.
\end{proof}

\begin{remark}\label{rem:an-curvature-special-cases}
Proposition~\ref{prop:an-ample-semi} includes positive Kobayashi
curvature and Griffiths-semipositive Hermitian metrics as
special cases. For positive Kobayashi curvature, Wu's result
\cite[Theorem~1 and the proof of Corollary~2]{Wu22} already gives
a smooth complex norm on $E^*$ with plurisubharmonic square,
without assuming ampleness in advance. In the Hermitian case,
the dual metric $h^*$ is Griffiths seminegative, so its squared
norm is directly plurisubharmonic: the vertical Levi form is
positive and the horizontal Levi form is nonnegative by
\eqref{eq:conv-kobayashi}. This observation also requires no
ampleness assumption. For our ample examples, the single
application of Proposition~\ref{prop:an-ample-semi} suffices.
\end{remark}

\section{Stable ample bundles on Hirzebruch surfaces}
\label{sec:hirzebruch-geometry}

We first establish an algebraic approximation theorem independent
of the analytic obstruction: a specified split bundle lies in the
$C^\infty$ closure of the ample slope-stable locus, with polarization
and Chern classes fixed. We then combine this theorem with the
convex obstruction to prove Theorem~\ref{thm:hirz-main}.

Let $e\geq0$, and let
\[
 X=\mathbb F_e=\mathbb P_{\mathbb P^1}
       \bigl(\mathcal O\oplus\mathcal O(e)\bigr),
 \qquad \pi:X\longrightarrow\mathbb P^1.
\]
Write $f$ for the class of a ruling fibre and $S$ for a section of
self-intersection $-e$. For $e=0$, choose either complementary
ruling section. Thus
$$S^2=-e,\quad S\cdot f=1,\quad f^2=0, \quad K_X=-2S-(e+2)f.$$
The effective cone is generated by $S$ and $f$. On the Hirzebruch surface $\mathbb F_e$, an integral divisor $aS+bf$ is ample if and only if it is very ample. These equivalent conditions hold precisely when $a>0$ and $b>ae$, or, equivalently, when $a\geq1$ and $b\geq ae+1$.
These descriptions and the ruling cohomology formulas below follow
from the standard theory of Hirzebruch surfaces
\cite[Chapter~V, Section~2]{Hartshorne77}.
We use the following data:
\begin{equation}\label{eq:geom-data}
 \begin{aligned}
 H&=S+(3e+1)f,& P&=4H-f,\\
 L_1&=\mathcal O_X(5S-f),&
 L_2&=\mathcal O_X\bigl(-S+(12e+5)f\bigr),\\
 E_0&=L_1\oplus L_2.
 \end{aligned}
\end{equation}
Both $H$ and $P$ are very ample. Direct calculation
gives
\begin{equation}\label{eq:geom-chern}
 H^2=5e+2,\qquad c_1(E_0)=4H,\qquad
 c_2(E_0)=65e+26=13H^2.
\end{equation}
Here and below $c_2$ is identified with its degree.

\begin{theorem}\label{thm:geom-closure}
For every $e\geq0$, the holomorphic bundle $E_0$ in
\eqref{eq:geom-data} is a $C^\infty$ limit of ample
$\mu_P$-stable holomorphic structures on its underlying smooth
complex vector bundle. More precisely, there are integrable
Dolbeault operators $\bar\partial_j$ on that fixed bundle such that
$\bar\partial_j\to\bar\partial_{E_0} \text{ in every finite }C^m\text{ norm},$
and every $(E_0^{\mathrm{sm}},\bar\partial_j)$ is ample and
$\mu_P$-stable. Each has the Chern classes in
\eqref{eq:geom-chern}. Here the superscript $\mathrm{sm}$ indicates that we retain only the underlying smooth complex vector bundle, forgetting its holomorphic structure.

\end{theorem}

We prove the theorem by constructing an ample prioritary bundle and
a stable prioritary bundle with these Chern classes, and then using
the irreducibility of the prioritary stack. The ample comparison
bundle need not itself be stable.

\subsection{An ample comparison bundle}

\begin{proposition}\label{prop:geom-ample}
There are linearly independent sections
$a_0,\ldots,a_3\in H^0(X,\mathcal O_X(H))$ with no common zero.
Fix such a quadruple. For a general quadruple
$b_0,\ldots,b_3\in H^0(X,\mathcal O_X(3H))$, the sequence
\begin{equation}\label{eq:geom-resolution}
 0\longrightarrow\mathcal O_X(-H)\oplus\mathcal O_X(-3H)
 \xrightarrow{\,(a,b)\,}\mathcal O_X^{\oplus4}
 \longrightarrow E_{\mathrm{amp}}\longrightarrow0
\end{equation}
is exact with a locally free, ample cokernel of rank two. Its Chern
classes are $c_1(E_{\mathrm{amp}})=4H$ and
$c_2(E_{\mathrm{amp}})=13H^2$.
\end{proposition}

\begin{proof}
Under our convention for projective bundles, $\mathcal{O}_X(S) \simeq \mathcal{O}_{\mathbb P (\mathcal O\oplus\mathcal O(e))}(1)$ and $\mathcal{O}_X(f) \simeq \pi^* \mathcal{O}_{\mathbb{P}^1}(1)$. Thus
\[
 \pi_*\mathcal O_X(H)
 =\mathcal O_{\mathbb P^1}(3e+1)
  \oplus\mathcal O_{\mathbb P^1}(2e+1),
 \quad h^0(X,\mathcal O_X(H))=5e+4\geq 4.
\]
Let $W=H^0(X,\mathcal O_X(H))^{\oplus4}$. Since $H$ is globally generated, requiring all four sections to vanish at a fixed point imposes four independent linear conditions on $W$. As the point varies over the surface $X$, the incidence variety of common zeros has dimension $\dim W-4+2=\dim W-2$. Its projection to $W$ is closed because $X$ is projective, and is therefore a proper closed subset. Thus a general quadruple has no common zero. Moreover, since $h^0(X,\mathcal O_X(H))\ge4$, linearly independent quadruples form a nonempty open subset of $W$, so a general quadruple satisfies both properties. Fix such a quadruple $a$ and put
$$V=H^0(X,\mathcal O_X(3H))^{\oplus4},\qquad N=\dim V.$$

At a fixed point $x$, the vector $a(x)$ is nonzero. The map
$(a,b)$ drops rank at $x$ precisely when $b(x)$ belongs to the
one-dimensional subspace spanned by $a(x)$, after choosing local
trivializations of the line bundles. The evaluation map for $b$ is surjective onto
$\mathbb C^4$, so this is a codimension-three condition on $V$.
The rank-drop incidence in $X\times V$ has dimension at most
$N-1$ and has closed image in $V$. Its complement is a nonempty
Zariski-open set; for $b$ in this complement,
\eqref{eq:geom-resolution} has locally free rank-two cokernel.

We use the convention that $\mathbb P(E_{\mathrm{amp}}^*)$ parametrizes one-dimensional subspaces of the dual fibres. Dualizing \eqref{eq:geom-resolution} gives
\[
0\longrightarrow E_{\mathrm{amp}}^*
\longrightarrow(\mathbb C^4)^*\otimes\mathcal O_X
\longrightarrow
\mathcal O_X(H)\oplus\mathcal O_X(3H)
\longrightarrow0,
\]
where the last map sends $(w_i)$ to
$\left(\sum_iw_i a_i,\ \sum_iw_i b_i\right).$
Thus $E_{\mathrm{amp}}^*$ is a subbundle of the trivial bundle with fibre $(\mathbb C^4)^*$. Viewing a line in $(E_{\mathrm{amp}}^*)_x$ as a line in $(\mathbb C^4)^*$ defines a morphism
$$q_b:\mathbb P(E_{\mathrm{amp}}^*) \longrightarrow\mathbb P((\mathbb C^4)^*)\simeq\mathbb P^3.$$
The tautological line bundles are compatible with this construction, so
$\mathcal O_{\mathbb P(E_{\mathrm{amp}}^*)}(1) \simeq q_b^*\mathcal O_{\mathbb P^3}(1).$
For $w=[w_0:\cdots:w_3]$, the fibre $q_b^{-1}(w)$ is naturally identified with the common zero scheme in $X$ of
\begin{equation}\label{eq:geom-fibre}
s_w=\sum_{i=0}^3w_i a_i
\in H^0(X,\mathcal O_X(H)),
\quad
t_w=\sum_{i=0}^3w_i b_i
\in H^0(X,\mathcal O_X(3H)).
\end{equation}
Indeed, a point $x$ belongs to this zero scheme precisely when the line represented by $w$ lies in $(E_{\mathrm{amp}}^*)_x$. Since the $a_i$ are linearly independent, $s_w$ is not identically zero for any $w\in\mathbb P^3$. We shall show that, for a general choice of $b$, the sections $s_w$ and $t_w$ have no common curve for any $w$. This will ensure that all fibres of $q_b$ are finite.

If a nonzero effective divisor $D$ is a subdivisor of
$\operatorname{div}(s_w)$, both $D$ and $H-D$ are effective. Their
divisor classes therefore satisfy
\begin{equation}\label{eq:geom-subdivisors}
 D\sim uS+vf,\quad u\in\{0,1\},\quad 0\leq v\leq3e+1,
 \quad (u,v)\neq(0,0).
\end{equation}
There are therefore only finitely many possible classes. For each
such class $D_0$, form the projective parameter scheme
$$T_{D_0}= \bigl\{(w,[u],[v])\in\mathbb P^3\times|D_0|\times|H-D_0|: [s_w]=[uv]\bigr\}.$$
Here $\left|D\right|:=\mathbb{P}(H^0\left(X, \mathcal{O}_X\left(D\right)\right))$ denotes the complete linear system associated to $D$.
Empty linear systems contribute no parameters. The product of two
nonzero sections is nonzero, so the displayed condition is a closed
condition between projective spaces. For fixed $w$, its points
correspond to effective subdivisors of the fixed divisor
$\operatorname{div}(s_w)$. There are finitely many such subdivisors,
including their possible multiplicities. It follows that
\begin{equation}\label{eq:geom-paramdim}
 \dim T_{D_0}\leq3.
\end{equation}

For every integral curve $C\subset X$, the image of the restriction
map
$H^0(X,\mathcal O_X(3H))\to H^0(C,\mathcal O_C(3H))$
has dimension at least four. Indeed, very ampleness of $H$ gives
two sections $r_0,r_1$ whose restrictions have nonconstant ratio
in the function field of $C$. The restrictions of
$r_0^3,r_0^2r_1,r_0r_1^2,r_1^3$ are then linearly independent:
otherwise the nonconstant ratio would satisfy a nonzero polynomial over
$\mathbb C$, and hence would be constant. Let $D$ be a nonzero effective divisor, and choose an integral curve $C$ in its support. Any section vanishing along $D$ also vanishes along $C$. The preceding estimate therefore shows that vanishing along $D$ imposes at least four independent linear conditions on $H^0(X,\mathcal O_X(3H))$.

For a fixed point $(w,D)$ of $T_{D_0}$, the linear map
$V\to H^0(X,\mathcal O_X(3H))$ sending $b$ to $t_w$ is
surjective. Thus the condition $t_w|_D=0$ has codimension at least four in $V$.
This condition defines a closed incidence subset over $T_{D_0}$. To see that this subset is closed
without a base-change assumption, use multiplication by the
universal section defining $D$: for $D=\operatorname{div}(\sigma)$, with $[\sigma]\in|D_0|$, the spaces
$$\sigma\cdot H^0(X,\mathcal O_X(3H-D_0)) \subset H^0(X,\mathcal O_X(3H))$$
form a subbundle of a trivial vector bundle over $|D_0|$, since
the universal multiplication map
$$H^0(X,\mathcal O_X(3H-D_0))\otimes\mathcal O_{|D_0|}(-1) \longrightarrow H^0(X,\mathcal O_X(3H))\otimes\mathcal O_{|D_0|}$$
is fibrewise injective and hence has constant rank. The excluded condition is the vanishing
of $t_w$ in the quotient bundle. By \eqref{eq:geom-paramdim}, the
incidence subset has dimension at most $N-4+3=N-1$. The projection to $V$ is
closed because $T_{D_0}$ is projective. Its image is consequently
a proper closed subset.

A general $b$ avoids these finitely many bad loci as well as the
rank-drop locus. For that $b$, all fibres in
\eqref{eq:geom-fibre} are zero-dimensional, hence finite. The
proper morphism $q_b$ is therefore finite. Pullback by a finite
morphism preserves ampleness, so
$\mathcal O_{\mathbb P(E_{\mathrm{amp}}^*)}(1)$ is ample.
Finally,
$$c(E_{\mathrm{amp}})=\frac{1}{(1-H)(1-3H)} =1+4H+13H^2,$$
as asserted.
\end{proof}

\subsection{Prioritariness of the two comparison points}

For $p\in\mathbb P^1$, let $f_p=\pi^{-1}(p)$ denote the corresponding fibre, viewed as a divisor on $X$. A torsion-free coherent sheaf $E$ is called \emph{prioritary for $\pi$} if
$\operatorname{Ext}^2_X(E,E(-f_p))=0$
for every $p\in\mathbb P^1,$
where $E(D):=E\otimes\mathcal O_X(D)$.
Since all points of $\mathbb P^1$ are linearly equivalent, so are their pullbacks $f_p$. Thus the vanishing condition is independent of $p$, and we write $f$ for the common fibre class.
For a locally free sheaf $E$, Serre duality gives
\[
\operatorname{Ext}^2_X(E,E(-f))^\vee
\simeq
\operatorname{Hom}_X(E(-f),E(K_X))
\simeq
\operatorname{Hom}_X(E,E(K_X+f)).
\]
Consequently, $E$ is prioritary if and only if
\begin{equation}\label{eq:geom-prioritary}
 \operatorname{Hom}(E,E(K_X+f))=0.
\end{equation}
Equivalently, there is no nonzero global morphism from $E$ to $E(K_X+f)$.

\begin{lemma}\label{lem:geom-prioritary}
Both $E_0$ and the bundle $E_{\mathrm{amp}}$ constructed above are
prioritary for $\pi$.
\end{lemma}

\begin{proof}
Put $D=K_X+f=-2S-(e+1)f$. By \eqref{eq:geom-prioritary}, it suffices to show that
$\operatorname{Hom}_X(E,E(D))=0$
for $E=E_0$ and $E=E_{\mathrm{amp}}$.

First, since $E_0=L_1\oplus L_2$, its endomorphism bundle decomposes as
\[
\mathcal E nd(E_0)(D)
\simeq
\mathcal O_X(D)^{\oplus2}
\oplus (L_1\otimes L_2^{-1})(D)
\oplus (L_2\otimes L_1^{-1})(D).
\]
The two off-diagonal summands are
$(L_1\otimes L_2^{-1})(D) \simeq\mathcal O_X(4S-(13e+7)f),$
and
$(L_2\otimes L_1^{-1})(D) \simeq\mathcal O_X(-8S+(11e+5)f).$
Each of these divisor classes, as well as $D$, has a negative coefficient in the $S,f$ basis. Since every effective class has nonnegative coefficients in this basis, none of these line bundles has a nonzero global section. Hence
$$\operatorname{Hom}_X(E_0,E_0(D)) = H^0(X,\mathcal E nd(E_0)(D)) =0.$$

For $E_{\mathrm{amp}}$, precomposition with the surjection
$q:\mathcal O_X^{\oplus4}\twoheadrightarrow E_{\mathrm{amp}}$
gives an injection
\[
\operatorname{Hom}_X(E_{\mathrm{amp}},E_{\mathrm{amp}}(D))
\hookrightarrow
H^0(X,E_{\mathrm{amp}}(D))^{\oplus4},
\qquad
\varphi\longmapsto\varphi\circ q.
\]
Thus it suffices to prove that $H^0(X,E_{\mathrm{amp}}(D))=0$.

Twisting \eqref{eq:geom-resolution} by $\mathcal O_X(D)$ gives
\[
0\longrightarrow
\mathcal O_X(D-H)\oplus\mathcal O_X(D-3H)
\longrightarrow
\mathcal O_X(D)^{\oplus4}
\longrightarrow
E_{\mathrm{amp}}(D)
\longrightarrow0.
\]
We already know that $H^0(X,\mathcal O_X(D))=0$.
For $j=1,3$, Serre duality gives
$H^1(X,\mathcal O_X(D-jH)) \simeq H^1(X,\mathcal O_X(jH-f))^*.$
The ruling cohomology formulas yield
\[
\pi_*\mathcal O_X(jH-f)
=
\bigoplus_{i=0}^{j}
\mathcal O_{\mathbb P^1}\bigl(j(3e+1)-1-ie\bigr),
\quad
R^1\pi_*\mathcal O_X(jH-f)=0.
\]
Every summand has degree at least
$j(2e+1)-1\ge0,$
and therefore has vanishing first cohomology on $\mathbb P^1$. By Leray and Serre duality, this proves
$H^1(X,\mathcal O_X(D-jH))=0$, $j=1,3.$

The long exact cohomology sequence now implies $H^0(X,E_{\mathrm{amp}}(D))=0$. The preceding injection then gives
$$\operatorname{Hom}_X(E_{\mathrm{amp}},E_{\mathrm{amp}}(D))=0.$$
Thus both bundles are prioritary.
\end{proof}

\subsection{A stable point with the same Chern classes}

\begin{proposition}\label{prop:geom-stable}
Every nonsplit extension
\begin{equation}\label{eq:geom-extension}
 0\longrightarrow L_1\longrightarrow V
 \longrightarrow L_2\longrightarrow0
\end{equation}
is $\mu_P$-stable and prioritary. Such extensions exist, and
$\dim\operatorname{Ext}^1(L_2,L_1)=105e+35.$
\end{proposition}

\begin{proof}
Since $L_2$ is a line bundle, the extensions in
\eqref{eq:geom-extension} are parametrized by
$$\operatorname{Ext}^1_X(L_2,L_1) \simeq H^1(X,L_2^\vee\otimes L_1) =H^1(X,\mathcal O_X(6S-(12e+6)f)).$$
The ruling cohomology formulas give
\[
\begin{gathered}
\pi_*\mathcal O_X(6S-(12e+6)f)
=
\bigoplus_{i=0}^{6}
\mathcal O_{\mathbb P^1}(-12e-6-ie),\\
R^1\pi_*\mathcal O_X(6S-(12e+6)f)=0.
\end{gathered}
\]
Thus, by Leray and the cohomology of line bundles on
$\mathbb P^1$,
$$\dim\operatorname{Ext}^1_X(L_2,L_1) = \sum_{i=0}^{6}(12e+5+ie) =105e+35>0.$$
In particular, nonsplit extensions exist. Fix one, denote its
nonzero extension class by $\xi$, and write
$q:V\twoheadrightarrow L_2$ for the quotient map.

A direct intersection calculation gives
\begin{equation}\label{eq:geom-slopes}
\mu_P(L_1)=40e+11,\quad
\mu_P(V)=40e+14,\quad
\mu_P(L_2)=40e+17.
\end{equation}
Since $X$ is a smooth surface and $V$ has rank two, it suffices
to test slope-stability on line-bundle subsheaves
$\iota:M\hookrightarrow V$; see \cite[Remark~2.2]{RRT22}.
Indeed, a rank-one torsion-free subsheaf has a line bundle as its
double dual, and its inclusion into $V$ extends across the finitely
many missing points. This extension remains injective and has
the same first Chern class. Its quotient need not be locally free.
Thus we need to show that
$\deg_P M<\mu_P(V)$
for every such $M$.

If $q\circ\iota=0$, then $\iota$ factors through $L_1$.
A nonzero map between line bundles gives
$\deg_P M\leq\deg_P L_1<\mu_P(V).$
Otherwise, $q\circ\iota:M\to L_2$ is a nonzero map of
line bundles. Its zero divisor $D$ is effective, and
$M\simeq L_2(-D).$
If $D=0$, this map is an isomorphism, giving a right inverse
to $q$ and splitting the extension. Hence $D\neq0$, and
$$\deg_P M =\deg_P L_2-D\cdot P =\mu_P(V)+3-D\cdot P.$$
It therefore remains to exclude the possibility
$D\cdot P\leq3$.

Every effective divisor class has the form $aS+bf$ with
$a,b\geq0$, and
$S\cdot P=8e+3,$ $f\cdot P=4.$
If $e\geq1$, every nonzero effective divisor has
$P$-degree at least four. Consequently,
$\deg_P M\leq\mu_P(V)-1<\mu_P(V),$
which proves stability in this case.

Suppose now that $e=0$. The only nonzero effective class
of $P$-degree at most three is $S$. Thus the only remaining
possibility is
$M\simeq L_2(-S),$ $\deg_P M=\mu_P(V).$
In this case, $q\circ\iota$ is multiplication by a nonzero
section $s\in H^0(X,\mathcal O_X(S))$,
$j_s:L_2(-S)\to L_2.$
The map $\iota$ would lift $j_s$ through $q$.
Such a lift exists precisely when the extension pulled back
along $j_s$ splits, or, equivalently, when
$j_s^*\xi=0$ in $\operatorname{Ext}^1_X(L_2(-S),L_1).$

Under the cohomological descriptions of these extension
spaces, the pullback map is multiplication by $s$:
\begin{equation}\label{eq:geom-F0-injection}
H^1(X,\mathcal O_X(6S-6f))
\xrightarrow{\ \cdot s\ }
H^1(X,\mathcal O_X(7S-6f)).
\end{equation}

To see that this map is injective, identify
$X=\mathbb F_0$ with $\mathbb P^1\times\mathbb P^1$,
ordering the factors so that
$\mathcal O_X(S)=\mathcal O(1,0)$ and
$\mathcal O_X(f)=\mathcal O(0,1)$.
By the K\"unneth formula,
\eqref{eq:geom-F0-injection} becomes
\[
\left[
H^0(\mathbb P^1,\mathcal O(6))
\xrightarrow{\ \cdot s\ }
H^0(\mathbb P^1,\mathcal O(7))
\right]
\otimes
\operatorname{Id}_{H^1(\mathbb P^1,\mathcal O(-6))}.
\]
The first factor is injective because multiplication by a
nonzero section cannot annihilate a nonzero section.
Tensoring with a vector space preserves injectivity.
Hence $j_s^*\xi\neq0$, since $\xi\neq0$.
This rules out such a lift and excludes the equality
case. Thus $V$ is also slope-stable when $e=0$.

Finally, tensoring with a line bundle preserves slope-stability,
so $V(K_X+f)$ is stable. Moreover,
$$\mu_P(V(K_X+f))-\mu_P(V) =(K_X+f)\cdot P =-20e-10<0.$$
There is no nonzero homomorphism from a semistable sheaf
of larger slope to one of smaller slope. Therefore
$\operatorname{Hom}_X(V,V(K_X+f))=0.$
By \eqref{eq:geom-prioritary}, $V$ is prioritary.
\end{proof}

\subsection{From an irreducible stack to smooth convergence}

\begin{proof}[Proof of Theorem~\ref{thm:geom-closure}]
Let $\mathcal P$ be the stack of torsion-free sheaves on $X$
that are prioritary for $\pi$ and have invariants
$(r,c_1,c_2)=(2,4H,13H^2).$
By Walter's result \cite[Proposition~2]{Walter}, this stack is
smooth and irreducible. The proposition applies to all prioritary
sheaves with these invariants; no semistability assumption is
required.

Let $\mathcal A$ and $\mathcal S$ be the loci in $\mathcal P$
parametrizing, respectively, ample vector bundles and
$\mu_P$-stable vector bundles. Both are open: local freeness,
ampleness, and slope-stability for a fixed polarization are
open conditions in families. For ampleness, apply
\cite[Theorem~1.2.17]{Lazarsfeld} to the tautological line bundle
on the relative projective bundle. For slope-stability, use
the openness of $0$-stability in
\cite[\S3.A, discussion following Definition~3.A.1]{HuybrechtsLehn};
there $0$-stability is precisely $\mu$-stability, and the
argument is the one used in Proposition~2.3.1 of that book.
One may replace $P$ by a very ample multiple without changing
$\mu_P$-stability.
The bundle constructed in Proposition~\ref{prop:geom-ample}
is prioritary by Lemma~\ref{lem:geom-prioritary}, so
$\mathcal A$ is nonempty. Proposition~\ref{prop:geom-stable}
similarly shows that $\mathcal S$ is nonempty.
Since $\mathcal P$ is irreducible,
$\mathcal U:=\mathcal A\cap\mathcal S$
is a nonempty dense open substack.

By Lemma~\ref{lem:geom-prioritary}, $E_0$ is a point of
$\mathcal P$. To describe nearby sheaves using an ordinary
algebraic parameter space, choose a smooth local chart
$q:T\to\mathcal P,$
where $T$ is of finite type over $\mathbb C$, together with
a complex point $t_0\in T$ representing $E_0$.
The map $q$ determines a family $\mathcal E$ on $X\times T$:
each parameter $t\in T$ corresponds to the sheaf
$\mathcal E_t:=\mathcal E|_{X\times\{t\}}$
on $X$. We fix an identification
$\mathcal E_{t_0}\simeq E_0$.
Since both $q$ and $\mathcal P$ are smooth, $T$ is smooth
over $\mathbb C$. Replacing $T$ by an open neighbourhood
of $t_0$, we may assume that $T$ is irreducible.
Moreover, the central fibre $E_0$ is locally free.
By openness of local freeness and projectivity of $X$,
we may shrink $T$ further so that $\mathcal E$ is a
vector bundle on all of $X\times T$.

Set $U_T=q^{-1}(\mathcal U)$. This is a dense open subset
of $T$. Indeed, if $T'\subset T$ is any nonempty open subset,
then $q(T')$ is nonempty and open because $q$ is smooth.
Irreducibility of $\mathcal P$ gives
$q(T')\cap\mathcal U\neq\varnothing$, and hence
$T'\cap U_T\neq\varnothing$.
The complement $T\setminus U_T$ is a proper closed algebraic
subset of the smooth irreducible variety $T$, so it has empty
interior in the complex-analytic topology. Thus, in a small
analytic coordinate ball $B$ centred at $t_0$, we can choose
$t_j\in U_T\cap B$, $t_j\to t_0.$

Choose smooth complex-linear identifications
$\Phi_t:E_0^{\mathrm{sm}} \xrightarrow{\sim}\mathcal E_t^{\mathrm{sm}},$ $\Phi_{t_0}=\operatorname{id},$
depending smoothly on $t\in B$. Such identifications are
obtained by choosing a smooth complex connection on
$\mathcal E|_{X\times B}$ and using parallel transport
along radial paths in $B$, with the point of $X$ fixed.
These identifications are smooth and need not be holomorphic.

Transport the fibrewise Dolbeault operators to the fixed
smooth bundle $E_0^{\mathrm{sm}}$ by setting
$\bar\partial_t = \Phi_t^{-1}\circ\bar\partial_{\mathcal E_t}\circ\Phi_t,$ $A_t=\bar\partial_t-\bar\partial_{E_0}.$
Then
$A_t\in \Omega^{0,1}\bigl(X,\operatorname{End}(E_0^{\mathrm{sm}})\bigr)$
depends smoothly on both $x$ and $t$, and $A_{t_0}=0$.
Since $X$ is compact, for every finite integer $m\geq0$,
$\|A_{t_j}\|_{C^m(X)}\to 0.$
Hence $\bar\partial_{t_j}\to\bar\partial_{E_0}$ in every
finite $C^m$ norm.

Each $\bar\partial_{t_j}$ defines a holomorphic bundle
isomorphic to $\mathcal E_{t_j}$, which is ample and
$\mu_P$-stable because $t_j\in U_T$. Its Chern classes are
the prescribed ones. Applying the dual identifications
gives the same convergence for the induced Dolbeault
operators on $E_0^*$.
\end{proof}

\subsection{The analytic obstruction and the Hirzebruch conclusion}

\begin{proof}[Proof of Theorem~\ref{thm:hirz-main}]
For the split model in \eqref{eq:geom-data}, consider the divisor
classes
\begin{equation}\label{eq:geom-test-divisors}
 A_1=6S+(6e+1)f,\qquad A_2=S+(13e+6)f.
\end{equation}
They are ample for every $e\geq0$: their coefficients satisfy
$6e+1>6e$ and $13e+6>e$. Moreover,
\[
 \begin{aligned}
 c_1(L_1)\cdot A_1
   &=(5S-f)\cdot(6S+(6e+1)f)=-1,\\
 c_1(L_2)\cdot A_2
   &=(-S+(12e+5)f)\cdot(S+(13e+6)f)=-1.
 \end{aligned}
\]
Choose K\"ahler representatives $\eta_s\in A_s$, for example
normalized curvature forms of positive Hermitian metrics on the
ample line bundles $\mathcal O_X(A_s)$. Put
$F_s=L_s^*$ and $F_0=F_1\oplus F_2=E_0^*$. Then
$\int_Xc_1(F_s)\wedge\eta_s=1>0,\qquad d\eta_s=0.$
Theorem~\ref{thm:an-lines}, with
$\Theta_s=\eta_s$ in complex dimension two, gives a
neighbourhood of the split holomorphic structure on $F_0$ in
which no smooth complex norm has plurisubharmonic square away
from zero.

By Theorem~\ref{thm:geom-closure}, choose an ample
$\mu_P$-stable structure on the smooth bundle underlying $E_0$
whose dual Dolbeault operator lies in that neighbourhood.
Denote the resulting holomorphic bundle by $E$. It has rank two
and the Chern classes in \eqref{eq:geom-chern}. By construction,
$E^*$ admits no smooth complex norm with plurisubharmonic square
away from zero. If $E$ admitted a smooth strongly pseudoconvex metric
with semipositive Kobayashi curvature, its ampleness and
Proposition~\ref{prop:an-ample-semi} would give such
a smooth norm, a contradiction. This also excludes positive
Kobayashi curvature and Griffiths-semipositive Hermitian metrics,
as noted in Remark~\ref{rem:an-curvature-special-cases}.
This proves the theorem.
\end{proof}

\begin{remark}\label{rem:geom-base-properties}
Every $\mathbb F_e$ is simply connected, by the homotopy sequence
of its $\mathbb P^1$-bundle over $\mathbb P^1$. For $e\geq1$, it
is not biholomorphic to a product of positive-dimensional compact
complex manifolds. Such a product would have to consist of two
simply connected compact curves, hence would be
$\mathbb P^1\times\mathbb P^1$. Every effective curve on that
product has nonnegative square, whereas the section $S$ on
$\mathbb F_e$ has square $-e<0$.

For $e=0$ and $e=1$ the surfaces are Fano: the ample-cone
criterion applies to
$-K_{\mathbb F_0}=2S+2f$ and $-K_{\mathbb F_1}=2S+3f$.
Moreover, $\mathbb F_0=\mathbb P^1\times\mathbb P^1$ is
homogeneous and admits a K\"ahler--Einstein metric of positive
Ricci curvature, obtained from equally normalized
Fubini--Study metrics. This assertion concerns the base and
places no extra condition on the metric of the vector bundle.
\end{remark}

For reference, the data for the first two examples are
\[
 \begin{array}{c|c|c|c|c}
 e&H&P&c_1(E)&c_2(E)\\\hline
 0&S+f&4S+3f&4S+4f&26\\
 1&S+4f&4S+15f&4S+16f&91
 \end{array}
\]
The final bundle is chosen in the stable ample locus approaching
$E_0$; it is neither the split bundle nor a specified ample
comparison quotient. The construction establishes existence
using local deformation families, without prescribing explicit
transition matrices for the final bundle.

\section{Two ampleness problems on Hirzebruch surfaces}
\label{sec:applications}

Fix an integer $e\ge0$ and put $B=\mathbb F_e$,
$H=S+(3e+1)f$, and $P=4H-f$. Choose the rank-two bundle
$E\to B$ constructed in the proof of Theorem~\ref{thm:hirz-main}.
It is ample and $\mu_P$-stable, with $c_1(E)=4H$ and
$c_2(E)=65e+26$. Its dual admits no smooth complex norm
with plurisubharmonic square away from zero, and $E$ admits no smooth strongly
pseudoconvex Finsler metric with semipositive Kobayashi curvature.
We use these properties to prove the two assertions of
Theorem~\ref{thm:applications-main}. The additional features
of the case $e=1$ are recorded in Remark~\ref{rem:applications-f1}.

The first application uses the curvature obstruction to construct a
single connected complex-analytic hypersurface that violates Fulton's
bound for every smooth strongly pseudoconvex Finsler metric,
answering Problem~\ref{prob:fulton} within this class.
The second application addresses the Hartshorne--Schneider
conjecture by passing from a concave neighbourhood to a dual
support norm, contradicting the norm obstruction.
Finally, the positive mean curvature of a Hermitian--Einstein
metric shows that the resulting complements are $3$-convex,
so their least smooth convexity index is exactly three.

\subsection{A fixed hypersurface for the Finsler formulation}

We recall the argument in Fulton's proof
\cite[p.~27]{Fulton87} in terms of the Chern horizontal
splitting of a Griffiths-positive Hermitian metric.
Put $r=\operatorname{rank}E$
and $k=\dim_{\mathbb C}\mathcal S$. At a nonzero vector $v$, write
$T_v^{1,0}E=H_{h,v}\oplus V_v$ for the Chern horizontal splitting.
For $G_h(v)=h(v,v)$, the $(1,0)$-differential $\partial G_h$
vanishes on $H_{h,v}$, and its vertical kernel is the complex
hyperplane $v^{\perp_h}$.
At a critical point of $G_h|_{\mathcal S}$, it follows that
\[
 T_v^{1,0}\mathcal S\subset\ker\partial G_h(v)
      =H_{h,v}\oplus v^{\perp_h},\quad
 \dim_{\mathbb C}(T_v^{1,0}\mathcal S\cap H_{h,v})\ge k-r+1.
\]
Griffiths positivity makes the Levi form negative definite
on $H_{h,v}$, giving the claimed bound on its restriction.
This is the complex Levi-form bound stated in
\cite[p.~26]{Fulton87}. It applies to every critical point,
without a nondegeneracy assumption, and the submanifold need
only be locally closed in the complex-analytic topology.

We now show that this conclusion need not hold when only
ampleness is assumed, even if one can choose a different smooth
strongly pseudoconvex Finsler metric on $E$ for each test
submanifold. We construct one connected submanifold on which
every metric in this class fails the bound.

\begin{proposition}\label{prop:app-fulton-test}
Let $E$ have rank $r\ge2$ over a connected smooth projective
$n$-fold $X$. Suppose that $E$ admits no smooth strongly
pseudoconvex Finsler metric with semipositive Kobayashi curvature.
There is a connected smooth Stein hypersurface
$\mathcal S\subset E^\circ$, locally closed in the
complex-analytic topology, such that every smooth strongly
pseudoconvex squared Finsler metric $G$ has a nondegenerate
critical point on $\mathcal S$ with at least $r$ positive
Levi eigenvalues.
In particular, its negative Levi index there is at most $n-1$.
\end{proposition}

\begin{proof}
For each metric, we construct a local graph that detects the failure
of semipositive Kobayashi curvature. We then combine countably
many such local graphs into one global graph.

\textbf{Choose a common graph domain.}
By GAGA \cite{Serre56}, $E$ is algebraic. A rational frame
of $E$ gives a trivialization over a nonempty Zariski-open subset of $X$.
Shrinking this subset, we obtain a nonempty affine open set
$U\subset X$ such that $E|_U\simeq U\times\mathbb C^r$.
Since $X$ is smooth and connected, it is irreducible.
Thus $U$ is dense and connected.
Write the fibre coordinates as $(w,t,s)$, where
$w,t\in\mathbb C$ and $s\in\mathbb C^{r-2}$, and put
$\mathcal B=U\times\mathbb C^*\times\mathbb C^{r-2}$.
This is a connected affine, hence Stein, manifold.
We will construct the required hypersurface as a holomorphic
graph $w=f(x,t,s)$ over $\mathcal B$.
The condition $t\ne0$ ensures that every such graph avoids
the zero section.

\textbf{Construct a local test for one metric.}
Fix a smooth strongly pseudoconvex squared Finsler metric $G$.
By hypothesis, its Kobayashi curvature is not semipositive.
Consequently, the horizontal Levi form of $G$ has a positive
direction at some point of $E^\circ$.

The existence of a positive horizontal direction is an open condition.
We may therefore choose such a point
$v_0=(x_0,w_0,t_0,s_0)$ with $x_0\in U$, $t_0\ne0$, and
$G_w(v_0)\ne0$. Indeed, the first two conditions define a
dense open subset, and $G_w$ cannot vanish on any nonempty
open set: otherwise $G_{w\bar w}$ would vanish there,
contrary to strong pseudoconvexity.

Let $V_{v_0}$ be the vertical tangent space and set
$K=\ker\partial G(v_0)$. The identities
in \eqref{eq:conv-euler} give
$$K=H_{G,v_0}\oplus(V_{v_0}\cap K),\quad \dim_{\mathbb C}(V_{v_0}\cap K)=r-1.$$
The two summands are Levi orthogonal. The Levi form is
positive definite on the second summand and has a positive
direction on the first. Its restriction to $K$ therefore
has at least $r$ positive eigenvalues.

Write $b_0=(x_0,t_0,s_0)\in\mathcal B$.
Since $G_w(v_0)\ne0$, the hyperplane $K$ is the graph of
a linear map over $T^{1,0}_{b_0}\mathcal B$.
We may thus choose a local holomorphic function $g$ whose
graph passes through $v_0$ and has tangent space $K$ there.
For $F=G\circ\operatorname{Graph}(g)$, we have
$\partial F(b_0)=0$. Hence $b_0$ is a critical point of $F$,
and its Levi form has at least $r$ positive eigenvalues.

We can make this critical point nondegenerate without changing
its Levi form. Choose local coordinates
$z=(z_1,\ldots,z_{n+r-1})$ on $\mathcal B$ centred at $b_0$,
and keep the value and first derivative of $g$ fixed.
The mixed derivatives $F_{\alpha\bar\beta}(b_0)$ then remain
unchanged, whereas
$F_{\alpha\beta}(b_0) =A_{\alpha\beta}+G_w(v_0)g_{\alpha\beta}(b_0)$,
where $A_{\alpha\beta}$ is fixed.
Because $G_w(v_0)\ne0$, the pure second derivatives of $F$
can be prescribed arbitrarily.

Choose the quadratic jet of $g$ so that the pure quadratic
part of $F$ is
$\tau\operatorname{Re}\sum_{\alpha=1}^{n+r-1}z_\alpha^2$.
The real Hessian of this term has eigenvalues $\pm2\tau$.
For sufficiently large $\tau$, it dominates the fixed mixed
part, making the full real Hessian invertible.
Thus the critical point is nondegenerate, while its Levi
form still has at least $r$ positive eigenvalues.

\textbf{Reduce to countably many robust local tests.}
Choose compact neighbourhoods $C'_G$ and $C_G$ of $b_0$
such that $b_0\in\operatorname{int}C'_G$,
$C'_G\subset\operatorname{int}C_G$, and $g$ is holomorphic
near $C_G$. We may take $C_G$ to be
$\mathcal O(\mathcal B)$-convex.
For example, embed the affine manifold $\mathcal B$ as
a closed submanifold of some $\mathbb C^N$ and intersect
it with a sufficiently small closed Euclidean ball
centred at $b_0$.
We also arrange that
$t(C_G)\subset\{|t-t_0|<|t_0|/4\}$.

Fix a smooth Hermitian metric on $E$, and measure the
$C^2$ distance between squared Finsler metrics on its compact
unit sphere bundle $\Sigma$.
By homogeneity, this controls their $C^2$ distance on any
fixed compact subset of $E^\circ$.
In particular, it controls this distance on a compact neighbourhood
of the graph over $C'_G$, which also contains all sufficiently
small perturbations of that graph.

The implicit function theorem implies that a nondegenerate
critical point persists under small $C^2$ perturbations.
By continuity, the Levi form also retains at least $r$
positive eigenvalues.
Moreover, Cauchy estimates convert uniform approximation
on $C_G$ into $C^2$ approximation on $C'_G$.
We therefore obtain an open metric neighbourhood
$\mathcal V_G$ and a number $\delta_G>0$ with the following
property: whenever $G'\in\mathcal V_G$ and $f$ is holomorphic
near $C_G$ with $\sup_{C_G}|f-g|<\delta_G$, the restriction
of $G'$ to the graph of $f$ has a nearby nondegenerate
critical point with at least $r$ positive Levi eigenvalues.

Let $\mathscr M$ be the space of all smooth strongly
pseudoconvex squared Finsler metrics, equipped with this
$C^2$ topology. Restriction to $\Sigma$ identifies
$\mathscr M$ with a subspace of the separable metrizable
space $C^2(\Sigma)$. Hence $\mathscr M$ is second countable
and therefore Lindel\"of.
The neighbourhoods $\mathcal V_G$ cover $\mathscr M$.
Choose a countable subcover
$\mathscr M=\bigcup_{j\ge1}\mathcal V_j$, retaining the
corresponding local graphs $g_j$, compact sets $C_j$,
and tolerances $\delta_j$.
Each local test now works simultaneously for every metric
in its associated neighbourhood $\mathcal V_j$.

\textbf{Move the local tests by fibre dilations.}
For $\lambda\in\mathbb C^*$, define
$D_\lambda(x,w,t,s)=(x,\lambda w,\lambda t,\lambda s)$
on $E|_U$, and
$d_\lambda(x,t,s)=(x,\lambda t,\lambda s)$ on $\mathcal B$.
The image of the graph of $g_j$ under $D_\lambda$
is the graph of
$g_j^\lambda=\lambda g_j\circ d_\lambda^{-1}$.
Every candidate metric satisfies
$G\circ D_\lambda=|\lambda|^2G$.
Thus dilation preserves critical points, nondegeneracy,
and the numbers of positive and negative Levi eigenvalues.
In particular, the local test remains valid for the same metric neighbourhood
$\mathcal V_j$.

The compact set becomes $d_\lambda(C_j)$, and the tolerance
becomes $|\lambda|\delta_j$.
Indeed, if a holomorphic function $h$ satisfies
$\sup_{d_\lambda(C_j)}|h-g_j^\lambda| <|\lambda|\delta_j$, then
$\widetilde h=\lambda^{-1}h\circ d_\lambda$ satisfies
$\sup_{C_j}|\widetilde h-g_j|<\delta_j$.
The local test therefore remains valid after dilation.

\textbf{Combine the tests by holomorphic approximation.}
We first record an elementary separation fact.
Suppose that $K$ and $L$ are holomorphically convex compact
sets in a Stein manifold $Y$, and that
$h\in\mathcal O(Y)$ maps them into disjoint closed discs
$\Delta_K,\Delta_L\subset\mathbb C$.
Then $K\cup L$ is holomorphically convex. To see this,
polynomial separation first forces its holomorphic hull to
lie over $\Delta_K\cup\Delta_L$. For a point over one disc
but outside the corresponding compact set, choose a holomorphic
function separating it from that compact set. Multiply this
function by polynomials in $h$ that approximate one on this
disc and zero on the other. One-variable Runge approximation
provides such polynomials, and the products separate the
point from $K\cup L$. This proves the claim.

Fix a holomorphically convex exhaustion $(L_j)_{j\ge0}$
of $\mathcal B$, with $L_j\subset\operatorname{int}L_{j+1}$.
We inductively construct global holomorphic functions $f_j$
and increasing holomorphically convex compact sets $K_j$.
Start with $f_0=0$ and $K_0=L_0$.

Suppose that $f_{j-1}$ and $K_{j-1}$ have been constructed.
Choose $M_{j-1}>0$ such that
$t(K_{j-1})\subset\{|t|\le M_{j-1}\}$.
The $t$-projection of the $j$th compact set lies in
a disc $\{|t-t_j|<|t_j|/4\}$ with $t_j\ne0$.
Choose $R_j>4M_{j-1}/3$ and dilate by
$\lambda_j=R_j/t_j$.
The new $t$-projection lies in
$\{|t-R_j|\le R_j/4\}$, which is disjoint from
$\{|t|\le M_{j-1}\}$.
Relabel the dilated data as $g_j,C_j,\delta_j$.
The separation fact, applied to the holomorphic coordinate
$t$, shows that $K_{j-1}\cup C_j$ is holomorphically convex.
On disjoint neighbourhoods of these two compact sets, prescribe
$f_{j-1}$ and $g_j$, respectively.
The Oka--Weil theorem
\cite[\S2.3]{Forstneric17}
then gives $f_j\in\mathcal O(\mathcal B)$ such that
$\sup_{K_{j-1}}|f_j-f_{j-1}|<\epsilon_j$ and
$\sup_{C_j}|f_j-g_j|<\epsilon_j$, where
$\epsilon_j=2^{-j-2}\min(1,\delta_1,\ldots,\delta_j)$.
Choose $K_j$ to be a sufficiently large member of the fixed
exhaustion so that
$K_{j-1}\cup C_j\cup L_j\subset\operatorname{int}K_j$.
This completes the induction.

The interiors of the $K_j$ exhaust $\mathcal B$, and
$\sum_j\epsilon_j<\infty$.
Hence $f_j$ converges locally uniformly to a holomorphic
function $f$ on $\mathcal B$.
For each fixed $j$, every subsequent approximation controls
the error on $C_j$, since $C_j\subset K_{m-1}$ for $m>j$.
Consequently,
\[
 \sup_{C_j}|f-g_j|
 \le \sum_{m\ge j}\epsilon_m
 \le 2^{-j-1}\delta_j
 <\delta_j.
\]
Thus the final graph retains the local test for every
metric in every $\mathcal V_j$.

\textbf{Verify the properties of the final graph.}
Set
$\mathcal S=\{(x,w,t,s)\in E|_U: t\ne0,\ w=f(x,t,s)\}$.
The graph map
$(x,t,s)\mapsto(x,f(x,t,s),t,s)$ identifies
$\mathcal S$ biholomorphically with $\mathcal B$.
Therefore $\mathcal S$ is a connected smooth Stein
hypersurface. It is closed in the open subset
$\{x\in U,\ t\ne0\}\subset E^\circ$, so it is locally
closed in the complex-analytic topology.

Every candidate metric belongs to some $\mathcal V_j$.
By the preceding approximation estimate and the defining
property of that local test, its restriction to $\mathcal S$
has a nondegenerate critical point with at least $r$
positive Levi eigenvalues.
Since $\dim_{\mathbb C}\mathcal S=n+r-1$, its negative
Levi index at that point is at most
$(n+r-1)-r=n-1$.
\end{proof}

\subsection{Concave neighbourhoods and smooth dual norms}

To address the Hartshorne--Schneider conjecture, we must exclude
arbitrary $2$-convex exhaustions, not only those arising from a
chosen bundle metric. The following lemma converts such an
exhaustion into a smooth dual norm. Its osculating-section
argument is related to \cite[Proposition~2.4 and Theorem~3.1]{Demailly99},
but the neighbourhood is not assumed to be a Finsler ball or to
have convex fibres.

\begin{lemma}\label{lem:app-tube}
Let $E\to M$ be an ample holomorphic vector bundle of rank $r$
over a compact complex manifold of dimension $n$. Suppose that an open
neighbourhood $\Omega$ of the zero section has compact closure
in $E$ and smooth boundary. Let $f$ be an outward defining
function, so $\Omega=\{f<0\}$ near the boundary and $df\ne0$
there. If the Levi form of $f$ has at least $n$ negative
eigenvalues on the complex boundary tangent space at every
boundary point, then $E^*$ admits a smooth complex norm whose
square is strongly real convex in each fibre and strictly
plurisubharmonic away from zero.
\end{lemma}

\begin{proof}
Fix a smooth Hermitian metric $k$ on $E$, and write
$k_{\mathbb R}:=\operatorname{Re}k$ for the underlying real
metric. All real gradients, covector norms, and unit normals
below are taken with respect to $k_{\mathbb R}$. We write
$d_vf$ and $D_v^2f$ for the ordinary real differential and
Hessian of the restriction of $f$ to a fibre.

\textbf{Construction of the support norm.}
Since the zero section is compact and contained in $\Omega$,
and $\overline\Omega$ is compact, we can choose constants
$0<a_0<b_0$ such that
$\{|v|_k\le a_0\}\subset\Omega$ and
$\overline\Omega\subset\{|v|_k<b_0\}$.
For each $x\in M$, put $A_x:=\overline\Omega\cap E_x$ and define
\begin{equation}\label{eq:app-support}
 p(x,\xi)=\max_{v\in A_x}|\xi(v)|,
 \qquad \xi\in E_x^*.
\end{equation}
The maximum exists because $A_x$ is compact. The linearity of
$\xi$ and the triangle inequality for the absolute value show
that $p(x,\cdot)$ is complex homogeneous and satisfies the
triangle inequality. Moreover, the two ball inclusions give
$a_0|\xi|_{k^*}\le p(x,\xi)\le b_0|\xi|_{k^*}$.
Thus $p(x,\cdot)$ is a complex norm on each fibre.
We do not yet know that $p$ varies continuously with $x$,
and we do not assume that it is smooth. We shall first prove
continuity and plurisubharmonicity, then establish a uniform
strong-convexity estimate for $Q:=p^2$. These properties will
allow us to construct a different, smooth norm.

\textbf{Boundary geometry at a support point.}
Fix $\xi_0\in E_{x_0}^*\setminus\{0\}$ and choose
$v_0\in A_{x_0}$ with
$|\xi_0(v_0)|=p(x_0,\xi_0)=:a$.
The lower bound for $p$ implies $a>0$. Also,
$v_0\in\partial\Omega$: otherwise $v_0$ would be an interior
point of the fibre domain, where the modulus of the nonconstant
linear function $\xi_0$ could not attain its maximum.

Choose a complex number $\lambda$ of modulus one such that
$\eta_0:=\lambda\xi_0$ satisfies $\eta_0(v_0)=a$, and put
$\ell:=\operatorname{Re}\eta_0$. Then
\begin{equation}\label{eq:app-support-plane}
 \ell(v)\le a\quad\text{for every }v\in A_{x_0},
 \qquad \ell(v_0)=a.
\end{equation}
Thus the affine real hyperplane $\{\ell=a\}$ supports
$A_{x_0}$ at $v_0$. Notice that the phase change does not alter
the support norm or its maximizing points.

Let $V:=T^{1,0}_{v_0}E_{x_0}$ and
$T:=\ker\partial f(v_0)$. Their complex dimensions are $r$
and $n+r-1$, respectively. By hypothesis, $T$ contains an
$n$-dimensional subspace on which the Levi form of $f$ is
negative definite.

We first prove that $d_vf(v_0)\ne0$. Suppose instead that
$d_vf(v_0)=0$. Then $V\subset T$, and the dimension formula
shows that an $n$-dimensional Levi-negative subspace of $T$
must meet $V$ nontrivially. Hence the vertical Levi form has
a negative direction. Restricting $f$ to the corresponding
complex affine line shows that its real vertical Hessian
has a negative direction as well: the sum of the two real
second derivatives on this line is negative.

The set of real vectors $w$ satisfying
$D_v^2f(v_0)(w,w)<0$ is nonempty, open, and invariant under
$w\mapsto-w$. Since $\ell$ is a nonzero real linear
functional, we may therefore choose such a $w$ with
$\ell(w)>0$. Taylor's formula and $d_vf(v_0)=0$ give
$f(v_0+tw)<0$ for all sufficiently small $t>0$.
But $\ell(v_0+tw)=a+t\ell(w)>a$, contradicting
\eqref{eq:app-support-plane}. This proves
$d_vf(v_0)\ne0$.

The real implicit function theorem now shows that the fibre
boundary is a smooth real hypersurface near $v_0$. Since
$\ell$ attains a maximum there at $v_0$, its differential
vanishes on the tangent space of this hypersurface.
Both $\ker d_vf(v_0)$ and $\ker d\ell$ have real codimension
one, so they coincide. Denote this common real tangent
hyperplane by $P_{\mathbb R}$. Moreover,
$d_vf(v_0)=c\,d\ell$ for some $c>0$: the sign follows because
the inward side $\{f<0\}$ lies in the supporting half-space
$\{\ell<a\}$ near $v_0$.

Furthermore, $f$ is nonnegative near $v_0$ on the affine
hyperplane $v_0+P_{\mathbb R}=\{\ell=a\}$.
Indeed, a point of this hyperplane with $f<0$ would be
interior to the fibre domain; a small displacement in a
direction where $\ell$ increases would then contradict
\eqref{eq:app-support-plane}.
Thus the restriction of $f$ to this affine hyperplane has a local
minimum at $v_0$, and $D_v^2f(v_0)$ is nonnegative on
$P_{\mathbb R}$.

Put $V_0:=V\cap T$. Since $f$ is real valued and
$d_vf(v_0)\ne0$, its vertical $(1,0)$-differential is a
nonzero complex linear functional. Consequently,
$\dim_{\mathbb C}V_0=r-1$.
For a vector in $V_0$, the corresponding real direction and
its $J$-rotate both lie in $P_{\mathbb R}$. The preceding
real Hessian estimate therefore gives
$\mathcal L_f|_{V_0}\ge0$.

Choose an $n$-dimensional Levi-negative subspace $W\subset T$.
Since the Levi form is nonnegative on $V_0$, we have
$W\cap V_0=\{0\}$. The kernel of the bundle projection
$d\pi:T\to T^{1,0}_{x_0}M$ is $V_0$. Hence $d\pi|_W$ is
injective and, since the dimensions agree, is an isomorphism.

\textbf{Supporting holomorphic sections and
plurisubharmonicity.}
Choose local holomorphic coordinates $z$ centred at $x_0$
and a holomorphic trivialization of $E$, writing
$v_0=(0,\nu_0)$. Since $W$ projects isomorphically onto the
base tangent space, it is the graph of a complex linear
map $A$. The local holomorphic section
$\sigma_0(z)=(z,\nu_0+Az)$ therefore satisfies
$\sigma_0(0)=v_0$ and
$d\sigma_0(T^{1,0}_0M)=W$.

Set $F_0:=f\circ\sigma_0$.
Because $W\subset\ker\partial f(v_0)$, the chain rule gives
$dF_0(0)=0$. Its Levi form at $0$ is the restriction of
$\mathcal L_f$ to $W$ and is therefore negative definite.
Its Taylor expansion has the form
$$F_0(z)=\sum_{\alpha,\beta}H_{\alpha\bar\beta} z_\alpha\bar z_\beta+2\operatorname{Re}P(z)+O(|z|^3),$$
where $(H_{\alpha\bar\beta})$ is negative definite and
$P$ is a homogeneous holomorphic quadratic polynomial.

The pure quadratic term can be removed without changing
the first jet of the section. Since
$\partial_vf(v_0)\ne0$, choose a constant vertical vector
$b$ in this trivialization with
$\partial_vf(v_0)(b)=1$, and set
$\sigma(z):=(z,\nu_0+Az-bP(z))$.
The correction has order two, so the value and tangent
space at the centre are unchanged. Its contribution to
the quadratic Taylor term of $f\circ\sigma$ is
$-2\operatorname{Re}P(z)$, which cancels the pure quadratic
part. Hence
$$f(\sigma(z))= \sum_{\alpha,\beta}H_{\alpha\bar\beta}z_\alpha\bar z_\beta +O(|z|^3).$$
After shrinking the coordinate neighbourhood, negative
definiteness gives
$f(\sigma(z))\le-\epsilon|z|^2$
for some $\epsilon>0$. Thus $\sigma(0)=v_0$ and
$\sigma(z)\in\Omega$ for every sufficiently small $z\ne0$.
In particular, $\sigma(z)\in A_z$ throughout this
neighbourhood.

We now prove continuity of $p$. Compactness of
$\overline\Omega$ gives upper semicontinuity: for a sequence
$(x_j,\xi_j)\to(x_0,\xi_0)$, choose maximizing vectors
$v_j\in A_{x_j}$. Pass first to a subsequence realizing the
upper limit and then to a convergent subsequence of $(v_j)$.
The limit lies in $A_{x_0}$. Continuity of the pairing yields
$$\limsup_j p(x_j,\xi_j)\le p(x_0,\xi_0).$$
For the reverse inequality, the section constructed above
gives $$p(z,\xi)\ge|\xi(\sigma(z))|,$$ with equality at
$(x_0,\xi_0)$. The right-hand side is continuous, so
$p$ is lower semicontinuous there. Hence $p$ is continuous
off the zero section. Its upper bound
$p(x,\xi)\le b_0|\xi|_{k^*}$ also gives continuity at zero.

The function $(z,\xi)\mapsto\xi(\sigma(z))$ is holomorphic
and nonzero near $(x_0,\xi_0)$. Therefore
$\log|\xi(\sigma(z))|$ is a local pluriharmonic function
bounded above by $\log p$, with equality at the centre.
Its mean-value property on every sufficiently small
holomorphic disc gives the submean inequality for
$\log p$ at that point. Since the point was arbitrary
and $\log p$ is continuous, $\log p$ is plurisubharmonic
on $(E^*)^\circ$.

It follows that $Q=p^2=\exp(2\log p)$ is plurisubharmonic
away from the zero section. The continuous extension of $Q$ is nonnegative
and equals zero on the zero section. At points of that
section the submean inequality is automatic. Thus $Q$
is plurisubharmonic on all of $E^*$.

\textbf{Uniform interior balls.}
Let $\mathcal K$ consist of the triples $(x,\xi,v)$ with
$|\xi|_{k^*}=1$, $v\in A_x$, and $|\xi(v)|=p(x,\xi)$.
The unit sphere bundle and $\overline\Omega$ are compact,
and the maximizing condition is closed by the continuity of
$p$. Hence $\mathcal K$ is compact.

We have proved that $d_vf$ is nonzero at every support
point. Consequently,
$m:=\min_{\mathcal K}|d_vf|_{k_{\mathbb R}^*}>0$.
Compactness and smoothness also give constants $\rho>0$
and $C\ge1$ such that, for every support point $v_0$ and
every $w\in E_x$ with $|w|_k\le\rho$, the point $v_0+w$
lies in the defining-function neighbourhood and
$$f(v_0+w)\le d_vf(v_0)(w)+C|w|_k^2.$$

Choose
$$0<\delta<\min\{a_0/2,\rho/2,m/(2C)\}.$$
At a support point $v_0$, let
$N_0:=\nabla_vf(v_0)/|\nabla_vf(v_0)|_k$ be the outward
real unit normal to the fibre boundary. By the definition
of the gradient,
$$d_vf(v_0)(w)=|d_vf(v_0)|_{k_{\mathbb R}^*} k_{\mathbb R}(w,N_0).$$

Consider the closed ball of radius $\delta$ centred at
$c:=v_0-\delta N_0$. A point $v_0+w$ belongs to this ball
precisely when $|w+\delta N_0|_k\le\delta$.
Expanding this inequality gives
$k_{\mathbb R}(w,N_0)\le-|w|_k^2/(2\delta)$, and the
triangle inequality gives $|w|_k\le2\delta<\rho$.
The uniform Taylor estimate therefore implies
$$f(v_0+w)\le-(m/(2\delta)-C)|w|_k^2\le0.$$
Thus $\overline B_k(c,\delta)\subset A_x$.

\textbf{The support norm as a supremum of ball support norms.}
For a centre $c$ obtained above, the support norm of
$\overline B_k(c,\delta)$ is
$$s_c(\xi):=\max_{|w|_k\le\delta}|\xi(c+w)| =|\xi(c)|+\delta|\xi|_{k^*}.$$
Since the ball is contained in $A_x$, we have
$s_c(\xi)\le p(x,\xi)$.

For each unit covector $\xi_0$, choose a maximizing point
$v_0$ and a phase multiple $\eta_0$ of $\xi_0$ satisfying
$\eta_0(v_0)=p(x,\xi_0)=:a$.
The affine hyperplane $\{\operatorname{Re}\eta_0=a\}$
supports $A_x$ at $v_0$ and is tangent to the fibre
boundary there. Hence its outward unit normal is $N_0$.

Because $|\eta_0|_{k^*}=1$, the real covector
$\operatorname{Re}\eta_0$ has norm one for
$k_{\mathbb R}$. Its metric dual is therefore $N_0$.
Thus $\operatorname{Re}\eta_0(N_0)=1$; together with
$|\eta_0(N_0)|\le1$, this gives $\eta_0(N_0)=1$.
It follows that
$\eta_0(c)=a-\delta>0$, since $a\ge a_0$ and
$\delta<a_0/2$. Consequently,
$s_c(\xi_0)=s_c(\eta_0)=a=p(x,\xi_0)$.

Every nonzero covector can be normalized to unit length.
By homogeneity, the preceding contact property proves
\begin{equation}\label{eq:app-support-balls}
 Q(x,\xi)=\sup_c s_c(\xi)^2,
\end{equation}
where $c$ ranges over the centres of all the supporting
balls in $E_x$. These centres satisfy
$|c|_k\le b_0+\delta$.

\textbf{Uniform strong real convexity.}
We next prove that all the functions $s_c^2$ have a common
strong-convexity constant. Fix a fibre, choose unitary
coordinates, and write $|\xi|$ for $|\xi|_{k^*}$.
All derivatives in this calculation are real derivatives
on this fibre.

The function $|\xi(c)|$ may fail to be smooth where
$\xi(c)=0$. To justify the Hessian calculation there,
introduce
$$s_{c,\tau}(\xi):= \sqrt{|\xi(c)|^2+\tau|\xi|^2}+\delta|\xi|,$$
where $0<\tau\le1$.
Its first summand is a Hermitian norm, so
$s_{c,\tau}$ is a norm that is smooth away from zero.
The bound on $|c|_k$ gives a common Lipschitz constant
$L_0:=\sqrt{(b_0+\delta)^2+1}+\delta$.
In particular,
$|ds_{c,\tau}(w)|\le L_0|w|$,
independently of $c$ and $\tau$.

For $\xi\ne0$, put $u:=\xi/|\xi|$ and decompose a real
direction as $w=\alpha u+b$, where $\alpha\in\mathbb R$
and $b$ is real-orthogonal to $u$.
Then $|w|^2=\alpha^2+|b|^2$.
The Euclidean norm satisfies
$D^2|\xi|(w,w)=|b|^2/|\xi|$.
Since the first summand of $s_{c,\tau}$ is convex, we obtain
$D^2s_{c,\tau}(w,w)\ge\delta|b|^2/|\xi|$.
Also, $s_{c,\tau}(\xi)\ge\delta|\xi|$.
Abbreviate $s:=s_{c,\tau}$ and $A:=ds_\xi(w)\in\mathbb R$.
Differentiating the square gives
$D^2(s^2)(w,w)=2A^2+2sD^2s(w,w)$, and hence
$$D^2(s^2)(w,w)\ge2A^2+2\delta^2|b|^2.$$

To control the radial component, use degree-one homogeneity:
$ds_\xi(u)=s(\xi)/|\xi|\ge\delta$.
Since $A=\alpha ds_\xi(u)+ds_\xi(b)$ and
$|ds_\xi(b)|\le L_0|b|$, we have
$|\alpha|\le(|A|+L_0|b|)/\delta$.
Therefore
$$|w|^2\le 2A^2/\delta^2+ (1+2L_0^2/\delta^2)|b|^2.$$
Set
$\gamma:=\delta^4/(\delta^2+2L_0^2)>0$.
Since $L_0\ge\delta$, we have $\gamma\le\delta^2/2$,
while
$\gamma(1+2L_0^2/\delta^2)=\delta^2$.
Multiplying the preceding estimate for $|w|^2$ by
$\gamma$ therefore gives
$\gamma|w|^2\le A^2+\delta^2|b|^2$.
Combining the two estimates yields
$$D^2(s_{c,\tau}^2)(w,w)\ge2\gamma|w|^2.$$

Thus $s_{c,\tau}^2-\gamma|\xi|^2$ has nonnegative real
Hessian away from zero. Smoothness off zero and degree-two
homogeneity imply that it extends as a $C^1$ function at
zero with differential zero. Its restriction to each real
line is therefore convex, including lines through zero.
Hence it is convex on the whole fibre.

As $\tau\downarrow0$, we have
$0\le s_{c,\tau}(\xi)-s_c(\xi)\le\sqrt{\tau}\,|\xi|$.
Thus $s_{c,\tau}^2-\gamma|\xi|^2$ converges locally
uniformly to $s_c^2-\gamma|\xi|^2$, which is consequently
convex. The constant $\gamma$ is independent of the fibre
and the centre. Finally, \eqref{eq:app-support-balls} gives
$$Q(x,\xi)-\gamma|\xi|_{k^*}^2 =\sup_c\bigl(s_c(\xi)^2-\gamma|\xi|_{k^*}^2\bigr).$$
A finite-valued supremum of convex functions is convex.
We have therefore proved that
$Q-\gamma|\xi|_{k^*}^2$ is fibrewise real convex.

\textbf{A smooth replacement.}
The continuous squared norm $Q$ is plurisubharmonic and
$Q-\gamma|\xi|_{k^*}^2$ is fibrewise convex. Applying
Lemma~\ref{lem:an-regularization} completes the proof.
\end{proof}
\begin{corollary}\label{cor:app-complement}
Let $E\to M$ be an ample holomorphic vector bundle of rank $r$
over a compact complex $n$-fold. If $E^*$ admits no smooth
complex norm with plurisubharmonic square on $(E^*)^\circ$, then
$\mathbb P(\mathcal O_M\oplus E)\setminus \mathbb P(\mathcal O_M)$
is not $r$-convex, where $r=\operatorname{rank}E$.
\end{corollary}

\begin{proof}
Write $Z=\mathbb P(\mathcal O_M\oplus E)$ and
$Y=\mathbb P(\mathcal O_M)$. The chart in which the
$\mathcal O_M$-coordinate is nonzero is the total space of
$E$, with $Y$ as its zero section. If $Z\setminus Y$ were
$r$-convex, its dimension $n+r$ would imply the existence of a smooth exhaustion
$\psi$ with at least $n+1$ positive Levi eigenvalues outside
a compact set $K$.

Choose a relatively compact disc neighbourhood $D\Subset E$
of the zero section. The set $Z\setminus D$ is compact and
disjoint from $Y$. Choose a regular value $c$ of $\psi$ greater
than its maximum on both $Z\setminus D$ and $K$. Since
$\{\psi\le c\}$ is compact in $Z\setminus Y$, the set
$\Omega=\{\psi>c\}\cup Y$
is an open neighbourhood of $Y$ with compact closure contained
in $D$ and smooth boundary. Its outward defining function
$f=c-\psi$ has at least $n+1$ negative Levi eigenvalues.
Restricting to the complex boundary hyperplane leaves at least
$n$ negative eigenvalues. Lemma~\ref{lem:app-tube} now gives
a smooth dual norm with plurisubharmonic square, a
contradiction. 
\end{proof}

\subsection{The exact convexity index}

The obstruction above rules out $2$-convexity. To show that
$3$-convexity holds, we use the positive mean curvature of a
Hermitian--Einstein metric. For a bundle $E\to X$, write
$U_E:=\mathbb P(\mathcal O_X\oplus E)\setminus\mathbb P(\mathcal O_X)$.

\begin{proposition}\label{prop:app-positive-mean}
Let $(X,\omega)$ be a compact K\"ahler $n$-fold and let
$E\to X$ have rank $r$. If $E$ admits a Hermitian metric $h$
whose mean curvature $\Lambda_\omega R_h$ is a positive
definite endomorphism, then $U_E$ is $(n+r-1)$-convex.
In particular, this holds when $E$ is slope-stable with
respect to $\omega$ and $\int_Xc_1(E)\wedge\omega^{n-1}>0$.
\end{proposition}

\begin{proof}
Put $G(v)=h(v,v)$ and $\rho=G^{-1}$. At $(x,v)\in E^\circ$,
choose a holomorphic frame that is unitary and Chern-normal
at $x$. For $u\in T_x^{1,0}X$ and $w\in E_x$, the Levi form
has no mixed horizontal--vertical terms and is given by
\begin{equation}\label{eq:app-reciprocal-levi}
 \mathcal L_\rho((u,w),\overline{(u,w)})
 =\frac{2\pi}{G^2}h(R_h(u,\bar u)v,v)
   -\frac{|w|_h^2}{G^2}
   +\frac{2|h(w,v)|^2}{G^3}.
\end{equation}
The horizontal block has positive $\omega$-trace, namely
$2\pi G^{-2}h((\Lambda_\omega R_h)v,v)$, and therefore has
at least one positive eigenvalue. The vertical block is
positive on $\mathbb Cv$: its value at $w=v$ is $G^{-1}$.
Thus $\rho$ has at least two positive Levi eigenvalues.
To obtain an exhaustion, identify $E$ with the affine chart
$[1:v]$ in $\mathbb P(\mathcal O_X\oplus E)$. Choose a smooth
function $\chi:[0,\infty)\to[0,1]$ that equals one near zero
and vanishes for large arguments. Then $\psi=\chi(G)/G$ on
$E^\circ$, extended by zero near the divisor at infinity, is
smooth on $U_E$. It tends to infinity near the removed zero
section, and its sublevel sets are compact because the ambient
projective bundle is compact. Outside a compact subset of
$U_E$, it agrees with $\rho$ and has at least two positive
Levi eigenvalues. Since $\dim U_E=n+r$, this proves
$(n+r-1)$-convexity.

For the last assertion, the Donaldson--Uhlenbeck--Yau theorem
\cite{Donaldson85,UhlenbeckYau86} gives a Hermitian--Einstein
metric with $\Lambda_\omega R_h=\lambda\Id_E$, where
$$\lambda=\frac{n\int_Xc_1(E)\wedge\omega^{n-1}}{r\int_X\omega^n}>0.$$ The first assertion applies.
\end{proof}

\begin{corollary}\label{cor:app-sharp-index}
For every bundle $E\to\mathbb F_e$ constructed in the proof of
Theorem~\ref{thm:hirz-main}, the complement $U_E$ is
$3$-convex but not $2$-convex. Its least smooth convexity index
is therefore three.
\end{corollary}

\begin{proof}
Choose a K\"ahler form $\omega$ with
$[\omega]=c_1(\mathcal O_{\mathbb F_e}(P))$.
Since $E$ is $\mu_P$-stable, it is slope-stable with respect
to $\omega$. Moreover,
$$\int_{\mathbb F_e}c_1(E)\wedge\omega
 =c_1(E)\cdot P=4H\cdot P>0,$$
because $H$ and $P$ are ample.
Proposition~\ref{prop:app-positive-mean}, with $n=r=2$,
therefore shows that $U_E$ is $3$-convex.
The dual-norm obstruction established in the proof of
Theorem~\ref{thm:hirz-main}, together with
Corollary~\ref{cor:app-complement}, excludes $2$-convexity.
Since $1$-convexity implies $2$-convexity, the least index is
exactly three.
\end{proof}

\begin{proof}[Proof of Theorem~\ref{thm:applications-main}]
Fix $e\ge0$ and choose the ample $\mu_P$-stable bundle
$E\to B=\mathbb F_e$ constructed in the proof of
Theorem~\ref{thm:hirz-main}. It has the stated Chern classes.
Moreover, $E$ admits no smooth strongly pseudoconvex Finsler
metric with semipositive Kobayashi curvature, and $E^*$ admits
no smooth complex norm with plurisubharmonic square away from zero.

\textbf{The smooth strongly pseudoconvex Finsler formulation.}
Apply Proposition~\ref{prop:app-fulton-test} with $n=r=2$.
It gives one fixed connected smooth Stein hypersurface
$\mathcal S\subset E^\circ$, locally closed in the complex-analytic
topology, with the required testing property.
Since the total space of $E$ has complex dimension four,
$\dim\mathcal S=3$. For every smooth strongly pseudoconvex
squared Finsler metric $G$ on $E$, the proposition supplies
a nondegenerate critical point of
$G|_{\mathcal S}$ with at least two positive Levi eigenvalues.
There can therefore be at most one negative eigenvalue.
This proves assertion~(\ref{item:applications-fulton}).

An extension of Fulton's lemma to this ample bundle would
instead require at least
$\dim\mathcal S-\operatorname{rank}E+1=3-2+1=2$
negative Levi eigenvalues at every critical point
\cite[p.~26]{Fulton87}. Thus no metric in this class satisfies
that bound, giving a negative answer to the smooth strongly
pseudoconvex Finsler formulation in Problem~\ref{prob:fulton}.
This also excludes Hermitian metrics, which are special cases
of smooth strongly pseudoconvex Finsler metrics. The hypersurface
$\mathcal S$ is independent of $G$, although the critical point
may depend on it. Only restrictions of squared bundle metrics
are tested; no conclusion is asserted for arbitrary smooth
functions on $\mathcal S$.

\textbf{The projective complement.}
Put $Z=\mathbb P(\mathcal O_B\oplus E)$ and
$Y=\mathbb P(\mathcal O_B)\simeq B$.
With our line convention, the affine chart where the
$\mathcal O_B$-coordinate is nonzero is identified with the total
space of $E$, and $Y$ becomes its zero section. The normal
directions along this section are precisely the fibres of $E$;
equivalently,
$N_{Y/Z}\simeq\operatorname{Hom}(\mathcal O_B,E)\simeq E$.
Thus $N_{Y/Z}$ is ample. Since $Z$ has dimension four and $Y$
has dimension two, the Hartshorne--Schneider conjecture would
make $Z\setminus Y$ $2$-convex
\cite[Conjecture~4.1]{Demailly99}. However,
Corollary~\ref{cor:app-sharp-index} shows that $Z\setminus Y$
is $3$-convex but not $2$-convex, giving the claimed counterexample
with least smooth convexity index three.

\textbf{Geometric properties.}
The base $B$ is rational, so $Y\simeq B$ is rational.
The algebraic bundle $E$ is trivial on a dense Zariski-open
subset $U\subset B$, so $Z|_U\simeq U\times\mathbb P^2$.
Therefore $Z$ is birational to $B\times\mathbb P^2$ and is
rational as well. The homotopy exact sequence of
$\mathbb P^1\to B\to\mathbb P^1$ gives $\pi_1(B)=0$.
Applying it again to $\mathbb P^2\to Z\to B$ gives
$\pi_1(Z)=0$, since both the fibre and the base are simply
connected. Thus $Y$ and $Z$ are simply connected.
For $e\ge1$, the surface $Y\simeq\mathbb F_e$ is not a
product by Remark~\ref{rem:geom-base-properties}. When $e=0$,
$Y\simeq\mathbb P^1\times\mathbb P^1$, and the other
conclusions remain unchanged. This completes
assertion~(\ref{item:applications-complement}).
\end{proof}

\noindent\textbf{Acknowledgments.}
The first named author is partially supported by the National
Key Research and Development Program of China (NKPs)
[Grant Number 2024YFA1013201], the Natural Science Foundation
of Chongqing (NSFCQ) [Grant Number CSTB2024NSCQ-LZX0040],
and the Special Project of Chongqing Municipal Science and
Technology Bureau [Grant Number 2025CCZ015].
The second named author  is sponsored by the National Key R\&D Program of China (Grant No. 2024YFA1013200) and the National Natural Science Foundation of China (Grant No. 12671100).

\noindent\textbf{AI Declaration}: The authors
used ChatGPT as an auxiliary tool in this work. The authors verified
and completed all mathematical arguments and take full responsibility for
the content of this paper.

\end{document}